\documentclass[12pt]{amsart}
\usepackage{graphicx}
\usepackage{amsmath,amsfonts}
\usepackage{amsthm,amssymb,latexsym}
\usepackage[active]{srcltx}

\newcommand{\stirling}[2]{\ifm {\left[{{#1}\atop {#2}}\right]} {$\left[{{#1}\atop {#2}}\right]$}}
\newtheorem{theorem}{Theorem}[section]
\newtheorem{corollary}[theorem]{Corollary}
\newtheorem{lemma}[theorem]{Lemma}

\newtheorem*{fact*}{Fact}
\newtheorem{proposition}[theorem]{Proposition}
\theoremstyle{definition}

\newtheorem{remark}[theorem]{Remark}
\numberwithin{equation}{section}

\newcommand{\N}{\mathbb N}
\newcommand{\Z}{\mathbb Z}

\newcommand{\A}{\mathbb A}
\newcommand{\F}{\mathbb F}
\newcommand{\K}{\mathbb K}
\newcommand{\R}{\mathbb R}
\newcommand{\Pp}{\mathbb P}

\newcommand{\bfs}{\boldsymbol}

\newcommand{\fq}{\F_{\hskip-0.7mm q}}

\newcommand{\cfq}{\overline{\F}_{\hskip-0.7mm q}}
\def\ifm#1#2{\relax \ifmmode#1\else#2\fi}

\newcommand{\klk}    {\ifm {,\ldots,} {$,\ldots,$}}

\newcommand{\plp}    {\ifm {+\cdots+} {$+\ldots+$}}

\makeatletter
\@namedef{subjclassname@2010}{%
  \textup{2010} Mathematics Subject Classification}
\makeatother

\begin{document}

\title[Value set of small families I]{On the
value set of small families of polynomials over a finite field,
I}
\author[E. Cesaratto et al.]{
%
Eda Cesaratto${}^{1,2}$,
%
Guillermo Matera${}^{1,2}$,
%
Mariana P\'erez${}^1$,
%
and Melina Privitelli${}^{2,3}$}

\address{${}^{1}$Instituto del Desarrollo Humano,
Universidad Nacional de Gene\-ral Sarmiento, J.M. Guti\'errez 1150
(B1613GSX) Los Polvorines, Buenos Aires, Argentina}
\email{\{ecesarat,\,gmatera,\,vperez\}@ungs.edu.ar}
\address{${}^{2}$ National Council of Science and Technology (CONICET),
Ar\-gentina}\email{mprivitelli@conicet.gov.ar}
\address{${}^{3}$Instituto de Ciencias,
Universidad Nacional de Gene\-ral Sarmiento, J.M. Guti\'errez 1150
(B1613GSX) Los Polvorines, Buenos Aires, Argentina}

\thanks{The authors were partially supported by the grants
PIP 11220090100421 CONICET and STIC--AmSud ``Dynalco''.}%

\begin{abstract}
We obtain an estimate on the average cardinality of the value set of
any family of monic polynomials of $\fq[T]$ of degree $d$ for which
$s$ consecutive coefficients $a_{d-1}\klk a_{d-s}$ are fixed. Our
estimate holds without restrictions on the characteristic of $\fq$
and asserts that $\mathcal{V}(d,s,\bfs{a})=\mu_d\,q+\mathcal{O}(1)$,
where $\mathcal{V}(d,s,\bfs{a})$ is such an average cardinality,
$\mu_d:=\sum_{r=1}^d{(-1)^{r-1}}/{r!}$ and $\bfs{a}:=(a_{d-1}\klk
d_{d-s})$. We provide an explicit upper bound for the constant
underlying the $\mathcal{O}$--notation in terms of $d$ and $s$ with
``good'' behavior. Our approach reduces the question to estimate the
number of $\fq$--rational points with pairwise--distinct coordinates
of a certain family of complete intersections defined over $\mathbb
F_q$. We show that the polynomials defining such complete
intersections are invariant under the action of the symmetric group
of permutations of the coordinates. This allows us to obtain
critical information concerning the singular locus of the varieties
under consideration, from which a suitable estimate on the number of
$\fq$--rational points is established.
\end{abstract}



\maketitle


\section{Introduction}
Let $\fq$ be the finite field of $q$ elements, 
let $T$ be an indeterminate over $\fq$ and let $f\in\fq[T]$. We
define the value set $\mathcal{V}(f)$ of $f$ as
$\mathcal{V}(f):=|\{f(c):c\in\fq\}|$ (cf. \cite{LiNi83}). Birch and
Swinnerton--Dyer established the following significant result
\cite{BiSD59}: for fixed $d\ge 1$, if $f$ is a generic polynomial of
degree $d$, then
$$\mathcal{V}(f)=\mu_d\,q+
\mathcal{O}(q^{1/2}),$$
where $\mu_d:=\sum_{r=1}^d{(-1)^{r-1}}/{r!}$ and the constant
underlying the $\mathcal{O}$--notation depends only on $d$.

Results on the average value $\mathcal{V}(d,0)$ of $\mathcal{V}(f)$
when $f$ ranges over all monic polynomials in $\fq[T]$ of degree $d$
with $f(0)=0$ were obtained by Uchiyama \cite{Uchiyama55a} and
improved by Cohen \cite{Cohen73}. More precisely, in \cite[\S
2]{Cohen73} it is shown that
$$\mathcal{V}(d,0)=\sum_{r=1}^d(-1)^{r-1}\binom{q}{r}q^{1-r}=\mu_d\,q
+\mathcal{O}(1).$$
However, if some of the coefficients of $f$ are fixed, the results
on the average value of $\mathcal{V}(f)$ are less precise. In fact,
Uchiyama \cite{Uchiyama55b} and Cohen \cite{Cohen72} obtain the
result that we now state. Let be given $s$ with $1\le s\le d-2$ and
$\bfs{a}:=(a_{d-1}\klk a_{d-s})\in\fq^s$. For every
$\boldsymbol{b}:=(b_{d-s-1}\klk b_1)$, let
$$f_{\boldsymbol{b}}:=f_{\boldsymbol{b}}^{\bfs{a}}
:=T^d+\sum_{i=1}^sa_{d-i}T^{d-i}+\sum_{i=s+1}^{d-1}b_{d-i}T^{d-i}.$$
Then for $p:=\mathrm{char}(\fq)>d$,
\begin{equation}\label{eq: average value set}
\mathcal{V}(d,s,\bfs{a}):=
\frac{1}{q^{d-s-1}}\sum_{\boldsymbol{b}\in\fq^{d-s-1}}\mathcal{V}(f_{\boldsymbol{b}})=
\mu_d\,q+\mathcal{O}(q^{1/2}),
\end{equation}
where the constant underlying the $\mathcal{O}$--notation depends
only on $d$ and $s$.

This paper is devoted to obtain an strengthened explicit version of
(\ref{eq: average value set}), which holds without any restriction
on $p$. More precisely, we shall show the following result (see
Theorem \ref{theorem: final main result} below).
\begin{theorem}
With notations as above, for $q>d$ and $1\le s\le \frac{d}{2}-1$ we
have
$$\left|\mathcal{V}(d,s,\bfs{a})-\mu_d\,q-\frac{1}{2e}\right|\le
\frac{(d-2)^5e^{2\sqrt{d}}}{2^{d-2}} +\frac{7}{q}.$$
\end{theorem}
This result strengthens (\ref{eq: average value set}) in several
aspects. The first one is that it holds without any restriction on
the characteristic $p$ of $\fq$, while (\ref{eq: average value set})
holds for $p>d$. The second aspect is that we show that
$\mathcal{V}(d,s,\bfs{a})=\mu_d\, q+\mathcal{O}(1)$, while (\ref{eq:
average value set}) only asserts that
$\mathcal{V}(d,s,\bfs{a})=\mu_d\, q+\mathcal{O}(q^{1/2})$. Finally,
we obtain an explicit expression for the constant underlying the
$\mathcal{O}$--notation with a good behavior, in the sense that we
prove that $\mathcal{V}(d,s,\bfs{a})=\mu_d\,q+\frac{1}{2e}+
\mathcal{O}(\rho^{-d})+\mathcal{O}(q^{-1})$ for any
$\frac{1}{2}<\rho<1$.

On the other hand, it must be said that our result holds for $s\le
d/2-1$, while (\ref{eq: average value set}) holds for $s$ varying in
a larger range of values. Numerical experimentation seems to
indicate that our result holds for any $s$ with $1\le s\le d-2$.
This aspect shall be addressed in a second paper, where we obtain an
explicit estimate showing that
$\mathcal{V}(d,s,\bfs{a})=\mu_d\,q+\mathcal{O}(q^{1/2})$ which is
valid for $1\le s\le d-3$ and $p>2$. We shall also exhibit estimates
on the second moment of the value set of the families of polynomials
under consideration.

In order to obtain our estimate, we express the quantity
$\mathcal{V}(d,s,\bfs{a})$ in terms of the number $\chi_r^{\bfs{a}}$
of certain ``interpolating sets'' with $d-s+1\le r\le d$ (see
Theorem \ref{theorem: interpolation problem - reduction to interp
sets} below). More precisely, for
$f_{\bfs{a}}:=T^d+a_{d-1}T^{d-1}\plp a_{d-s}T^{d-s}$, we define
$\chi_r^{\bfs{a}}$ as the number of $r$--element subsets of $\fq$ at
which $f_{\bfs{a}}$ can be interpolated by a polynomial of degree at
most $d-s-1$.

Then we express $\chi_r^{\bfs{a}}$ in terms of the number of
$q$--rational solutions with pairwise--distinct coordinates of a
polynomial system $\{R_{d-s}^{\bfs{a}}=0\klk R_{r-1}^{\bfs{a}}=0\}$,
where $R_{d-s}^{\bfs{a}}\klk R_{r-1}^{\bfs{a}}$ are certain
polynomials in $\fq[X_1,\ldots,X_r]$. A critical point for our
approach is that $R_{d-s}^{\bfs{a}}\klk R_{r-1}^{\bfs{a}}$ are
symmetric polynomials, namely invariant under any permutation of the
variables $X_1,\ldots, X_r$. More precisely, we prove that each
$R_j^{\bfs{a}}$ can be expressed as a polynomial in the first $s$
elementary symmetric polynomials of $\fq[X_1,\ldots, X_r]$
(Proposition \ref{prop: formula for Rj}). This allows us to
establish a number of facts concerning the geometry of set
$V_r^{\bfs{a}}$ of solutions of such a polynomial system (see, e.g.,
Corollary \ref{coro:dimension singular locus Vf} and Theorems
\ref{theorem: Vf at infinity} and \ref{theorem: proj closure of Vf
is abs irred}). Combining these results with estimates on the number
of $q$--rational points of singular complete intersections of
\cite{CaMaPr13}, we obtain our main result.

We finish this introduction by stressing on the methodological
aspects. As mentioned before, a key point is the invariance of the
family of sets $V_r^{\bfs{a}}$ under the action of the symmetric
group of $r$ elements. In fact, our results on the geometry of
$V_r^{\bfs{a}}$ and the estimates on the number of $q$--rational
points can be extended \textit{mutatis mutandis} to any symmetric
complete intersection whose projection on the set of primary
invariants (using the terminology of invariant theory) defines a
nonsingular complete intersection. This might be seen as a further
source of interest of our approach, since symmetric polynomials
arise frequently in combinatorics, coding theory and cryptography
(for example, in the study of deep holes in Reed--Solomon codes,
almost perfect nonlinear polynomials or differentially uniform
mappings; see, e.g., \cite{CaMaPr12}, \cite{Rodier09} or
\cite{AuRo09}).
%
%
\section{Value sets in terms of interpolating sets}
Let notations and assumptions be as in the previous section. In this
section we fix $s$ with $1\le s\le d-2$, an $s$--tuple
$\boldsymbol{a}:=(a_{d-1}\klk a_{d-s})\in\fq^s$ and denote
$$f_{\bfs{a}}:=T^d+a_{d-1}T^{d-1}\plp a_{d-s}T^{d-s}.$$
For every $\bfs{b}:=(b_{d-s-1}\klk b_1)\in\fq^{d-s-1}$, we denote by
$f_{\bfs{b}}:=f_{\bfs{b}}^{\bfs{a}}\in\fq[T]$ the following
polynomial
$$f_{\bfs{b}}:=f_{\bfs{a}}+
b_{d-s-1}T^{d-s-1}\plp b_1T.$$

For a given $\bfs{b}\in\fq^{d-s-1}$, the value set
$\mathcal{V}(f_{\bfs{b}})$ of $f_{\bfs{b}}$ equals the number of
elements $b_0\in\fq$ for which the polynomial $f_{\bfs{b}}+b_0$ has
at least one root in $\fq$. Let $\fq[T]_d$ denote the set of
polynomials of $\fq[T]$ of degree at most $d$, let
$\mathcal{N}:\fq[T]_d\to\Z_{\ge 0}$ be the counting function of the
number of roots in $\fq$ and let
$\bfs{1}_{\{\mathcal{N}>0\}}:\fq[T]_d\to\{0,1\}$ be the
characteristic function of the set of elements of $\fq[T]_d$ having
at least one root in $\fq$. From our previous assertion we deduce
the following identity:
\begin{eqnarray*}
\sum_{\bfs{b}\in\fq^{d-s-1}}\mathcal{V}(f_{\bfs{b}})&=&
\sum_{b_0\in\fq}\sum_{\bfs{b}\in\fq^{d-s-1}}
\bfs{1}_{\{\mathcal{N}>0\}}
(f_{\bfs{b}}+b_0)\\&=&\big|\{g\in\fq[T]_{d-s-1}:
\mathcal{N}(f_{\bfs{a}}+g)>0 \}\big|.
\end{eqnarray*}

For a set $\mathcal{X}\subseteq\fq$, we define
$\mathcal{S}_{\mathcal{X}}^{\bfs{a}}$ as the set $\fq[T]_{d-s-1}$ of
polynomials of $\fq[T]$ of degree at most $d-s-1$ which interpolate
$-f_{\bfs{a}}$ at all the points of $\mathcal{X}$, namely
$$
\mathcal{S}_{\mathcal{X}}^{\bfs{a}}:=\{g\in\fq[T]_{d-s-1}:
(f_{\bfs{a}}+g)(x)=0\ \textit{for\ any\ }x\in\mathcal{X}\}.
$$
Finally, for $r\in\N$ we shall use the symbol $\mathcal{X}_r$ to
denote a subset of $\fq$ of $r$ elements.
\begin{theorem}\label{theorem: interpolation problem - reduction to interp sets}
Let be given $s,d\in\N$ with $d<q$ and $1\le s\le d-2$. Then we have
\begin{equation}\label{eq: our formula for value sets}
\mathcal{V}(d,s,\bfs{a})= \sum_{r=1}^{d-s}(-1)^{r-1}\binom {q} {r}
q^{1-r}+\frac{1}{q^{d-s-1}}\hskip-0.25cm\sum_{r=d-s+1}^{d}
\hskip-0.25cm (-1)^{r-1} \chi_r^{\bfs{a}},\end{equation}
where $\mathcal{V}(d,s,\bfs{a})$ is defined as in (\ref{eq: average
value set}) and $\chi_r^{\bfs{a}}$ is the number of subsets
$\mathcal{X}_r$ of $\fq$ of $r$ elements such that there exists
$g\in\fq[T]_{d-s-1}$ for which
$(f_{\bfs{a}}+g)|_{\mathcal{X}_r}\equiv 0$ holds.
\end{theorem}
\begin{proof}
Given a subset $\mathcal{X}_r:=\{x_1\klk x_r\}\subset\fq$, we
consider the corresponding set
$\mathcal{S}_{\mathcal{X}_r}^{\bfs{a}} \subset\fq[T]_{d-s-1}$
defined as above. It is easy to see that
$\mathcal{S}_{\mathcal{X}_r}^{\bfs{a}}=
\bigcap_{i=1}^{r}\mathcal{S}_{\{x_i\}}^{\bfs{a}}$ and
$$
\{g\in\fq[T]_{d-s-1}: \mathcal{N}(f_{\bfs{a}}+g)>0
\}=\bigcup_{x\in\fq} \mathcal{S}_{\{x\}}^{\bfs{a}}.
$$
Therefore, by  the inclusion--exclusion principle we obtain
\begin{equation}\label{eq: eta(n,f) by inclusion-exclusion}
\mathcal{V}(d,s,\bfs{a})=\frac{1}{q^{d-s-1}}\bigg|\bigcup_{x\in
\fq}\mathcal{S}_{\{x\}}^{\bfs{a}}\bigg|=\frac{1}{q^{d-s-1}}
\sum_{r=1}^{q}(-1)^{r-1}\sum_{\mathcal{X}_r \subseteq \fq}\left|
\mathcal{S}_{\mathcal{X}_r}^{\bfs{a}}\right|.
\end{equation}

Now we estimate $|\mathcal{S}_{\mathcal{X}_r}^{\bfs{a}}|$ for a
given set $\mathcal{X}_r:=\{x_1\klk x_r\}\subset\fq$. Let
$g:=b_{d-s-1}T^{d-s-1}+\ldots+b_1T+b_0$ be an arbitrary element of
$\mathcal{S}_{\mathcal{X}_r}^{\bfs{a}}$. Then we have
$f_{\bfs{a}}(x_i)+g(x_i)=0$ for $1\le i\le r$. These identities can
be expressed in matrix form as follows:
$$\mathcal{M}(\mathcal{X}_r)\cdot\widehat{\bfs{b}}+f_{\bfs{a}}(\mathcal{X}_r)=0$$
where $\mathcal{M}(\mathcal{X}_r):=(m_{i,j})\in \fq^{r\times(d-s)}$
is the Vandermonde matrix defined by $m_{i,j}:=x_i^{d-s-j}$ for
$1\le i\le r$ and $1\le j\le d-s$, $\widehat{\bfs{b}}:=(b_{d-s-1}
\klk b_0)\in\fq^{d-s}$ and
$f_{\bfs{a}}(\mathcal{X}_r):=(f_{\bfs{a}}(x_1)\klk f_{\bfs{a}}(x_r))
\in \fq^r$.

Since $x_i\not= x_j$ for $i\not=j$, it follows that
\begin{equation}\label{eq: rank nonsquare Vandermonde matrix}
\mathrm{rank}(\mathcal{M}(\mathcal{X}_r))=\min\{r,d-s\}.
\end{equation}
We conclude that $\mathcal{S}_{\mathcal{X}_r}^{\bfs{a}}$ is an
$\fq$--linear variety and either
$\mathcal{S}_{\mathcal{X}_r}^{\bfs{a}}=\emptyset$ or
\begin{equation}\label{eq: dimension S(X_r)(f)}
\mathrm{rank}(\mathcal{M}(\mathcal{X}_r))+\dim
\mathcal{S}_{\mathcal{X}_r}^{\bfs{a}}=d-s.
\end{equation}

Suppose first that $r\leq d-s$. Then (\ref{eq: rank nonsquare
Vandermonde matrix}) implies
$\mathrm{rank}(\mathcal{M}(\mathcal{X}_r))=r$, and hence,
$\mathcal{S}_{\mathcal{X}_r}^{\bfs{a}}$ is not empty. From (\ref{eq:
dimension S(X_r)(f)}) one obtains
$\dim\mathcal{S}_{\mathcal{X}_r}^{\bfs{a}}=d-s-r$ and then
\begin{equation}\label{eq: S(X_r)(f) for r le n}
 |\mathcal{S}_{\mathcal{X}_r}^{\bfs{a}}|=q^{d-s-r}.
\end{equation}

Next we suppose that $r\ge d-s+1$. On one hand, if
$\mathcal{S}_{\mathcal{X}_r}^{\bfs{a}}$ is nonempty, then (\ref{eq:
dimension S(X_r)(f)}) implies $\dim
\mathcal{S}_{\mathcal{X}_r}^{\bfs{a}}=0$, and hence
$|\mathcal{S}_{\mathcal{X}_r}^{\bfs{a}}|=1$. On the other hand, if
$\mathcal{S}_{\mathcal{X}_r}^{\bfs{a}}$ is empty, then
$|\mathcal{S}_{\mathcal{X}_r}^{\bfs{a}}|=0$.

For $r>d$ we have that, if $g\in
\mathcal{S}_{\mathcal{X}_r}^{\bfs{a}}$, then $g\in \fq[T]_{d-s-1}$
and $f_{\bfs{a}}(x_i)+g(x_i)=0$ holds for $1\le i\le r$. As a
consequence, the (nonzero) polynomial $f_{\bfs{a}}+g$ has degree $d$
and $r$ different roots, which contradicts the hypothesis $r>d$. We
conclude that $\mathcal{S}_{\mathcal{X}_r}^{\bfs{a}}$ is empty, and
thus,
\begin{equation}\label{eq: S(X_r)(f) for r ge k>n}
|\mathcal{S}_{\mathcal{X}_r}^{\bfs{a}}|=0.
\end{equation}
Finally, for $d-s+1\le r\le d$ any of the cases
$|\mathcal{S}_{\mathcal{X}_r}^{\bfs{a}}|=0$ or
$|\mathcal{S}_{\mathcal{X}_r}^{\bfs{a}}|=1$ can arise.

Now we are able to obtain the expression for
$\mathcal{V}(d,s,\bfs{a})$ of the statement of the theorem.
Combining (\ref{eq: eta(n,f) by inclusion-exclusion}), (\ref{eq:
S(X_r)(f) for r le n}) and (\ref{eq: S(X_r)(f) for r ge k>n}) we
obtain
$$
q^{d-s-1}\mathcal{V}(d,s,\bfs{a})=\sum_{r=1}^{d-s} (-1)^{r-1} \binom
{q} {r}
    q^{d-s-r}+\sum_{r=d-s+1}^d(-1)^{r-1} \sum_{\mathcal{X}_r\subset\fq}
    |\mathcal{S}_{\mathcal{X}_r}^{\bfs{a}}|.
$$
From this identity we immediately deduce the statement of the
theorem.
\end{proof}
By definition we have $0\leq \chi_r^{\bfs{a}}\leq {\binom {q} {r}}$.
Our main result (Theorem \ref{theorem: estimate chi(a,r)}) asserts
that $\chi_r^{\bfs{a}}=\frac{1}{r!}q^{d-s}+\mathcal{O}(q^{d-s-1})$,
with an explicit upper bound for the constant underlying the
$\mathcal{O}$--notation in terms of $d$, $s$ and $r$.
%
%
%
\subsection{An algebraic approach to estimate the number of interpolating sets}
\label{subsec: relating interpol to rat points}
According to Theorem \ref{theorem: interpolation problem - reduction
to interp sets}, the asymptotic behavior of
$\mathcal{V}(d,s,\bfs{a})$ is determined by that of
$\chi_r^{\bfs{a}}$ for $d-s+1\le r\le d$. In order to find the
latter, we follow an approach inspired in \cite{ChMu07}, and further
developed in \cite{CaMaPr12}, which we now describe.

Fix a set $\mathcal{X}_r:=\{x_1\klk x_r\}\subset\fq$ of $r$ elements
and $g\in\fq[T]_{d-s-1}$. Then $g$ belongs to
$\mathcal{S}_{\mathcal{X}_r}^{\bfs{a}}$ if and only if
$(T-x_1)\cdots(T-x_r)$ divides $f_{\bfs{a}}+g$ in $\fq[T]$. Since
$\deg g\le d-s-1<r$, we have that the latter is equivalent to the
condition that $-g$ is the remainder of the division of
$f_{\bfs{a}}$ by $(T-x_1)\cdots(T-x_r)$. In other words, the set
$\mathcal{S}_{\mathcal{X}_r}^{\bfs{a}}$ is not empty if and only if
the remainder of the division of $f_{\bfs{a}}$ by $(T-x_1)\cdots
(T-x_r)$ has degree at most $d-s-1$.

Let $X_1,\ldots, X_r$ be indeterminates over $\cfq$, let
$\bfs{X}:=(X_1\klk X_r)$ and let $Q\in \fq[\bfs{X}][T]$ be the
polynomial
$$Q=(T-X_1)\cdots(T-X_r).$$
We have that there exists $R_{\bfs{a}}\in\fq[\bfs{X}][T]$ with $\deg
R_{\bfs{a}}\leq r-1$ such that the following relation holds:
\begin{equation}\label{eq: congruence mod Q}
f_{\bfs{a}}\equiv R_{\bfs{a}}\mod{Q}.
\end{equation}
Let $R_{\bfs{a}}:=R_{r-1}^{\bfs{a}}(\bfs{X})T^{r-1}\plp
R_0^{\bfs{a}}(\bfs{X})$. Then $R_{\bfs{a}}(x_1\klk x_r,T)\in\fq[T]$
is the remainder of the division of $f_{\bfs{a}}$ by $(T-x_1)\cdots
(T-x_r)$. As a consequence, the set
$\mathcal{S}_{\mathcal{X}_r}^{\bfs{a}}$ is not empty if and only if
the following identities hold:
\begin{equation}\label{eq: system defining Vr}
R_j^{\bfs{a}}(x_1\klk x_r)=0\quad (d-s\le j\le r-1).
\end{equation}
On the other hand, suppose that there exists $\mathbf{x}:=(x_1\klk
x_r)\in\fq^r$ with pairwise--distinct coordinates such that
(\ref{eq: system defining Vr}) holds and set
$\mathcal{X}_r:=\{x_1\klk x_r\}$. Then the remainder of the division
of $f_{\bfs{a}}$ by $Q(\bfs{x},T)=(T-x_1)\cdots (T-x_r)$ is a
polynomial $r_{\bfs{a}}:=R_{\bfs{a}}(\bfs{x},T)$ of degree at most
$d-s-1$. This shows that $\mathcal{S}_{\mathcal{X}_r}^{\bfs{a}}$ is
not empty. We summarize the conclusions of the argumentation above
in the following result.
\begin{lemma}\label{lemma: system defining Vr}
Let $s,d\in\N$ with $1\le s\le d-2$, let $R_j^{\bfs{a}}$ $(d-s\le
j\le r-1)$ be the polynomials of (\ref{eq: system defining Vr}) and
let $\mathcal{X}_r:=\{x_1\klk x_r\}\subset\fq$ be a set with $r$
elements. Then $\mathcal{S}_{\mathcal{X}_r}^{\bfs{a}}$ is not empty
if and only if (\ref{eq: system defining Vr}) holds.
\end{lemma}

It follows that the number $\chi_r^{\bfs{a}}$ of sets
$\mathcal{X}_r\subset \fq$ of $r$ elements such that
$\mathcal{S}_{\mathcal{X}_r}^{\bfs{a}}$ is not empty equals the
number of points $\bfs{x}:=(x_1,\ldots, x_r)\in\fq^r$ with
pairwise--distinct coordinates satisfying (\ref{eq: system defining
Vr}), up to permutations of coordinates, namely $1/r!$ times the
number of solutions $\bfs{x}\in\fq^r$ of the following system of
equalities and non-equalities:
\begin{equation}\label{eq: system rational solutions Vr}
R_j^{\bfs{a}}(X_1,\ldots, X_r)=0\quad(d-s\le j\le r-1),\ \prod_{1\le
i<j\le r}(X_i-X_j)\not=0.
\end{equation}
%
%
\subsection{$R_{\bfs{a}}$ in terms of the elementary symmetric polynomials}
Fix $r$ with $d-s+1\le r\le d$. Assume that $2(s+1)\le d$ holds and
consider the elementary symmetric polynomials $\Pi_1 ,\ldots, \Pi_r$
of $\fq[X_1,\ldots, X_r]$. For convenience of notation, we shall
denote $\Pi_0:= 1$. In Section \ref{subsec: relating interpol to rat
points} we obtain polynomials $R_j^{\bfs{a}}\in \fq[X_1,\ldots,
X_r]$ ($d-s\le j\le r-1$) with the following property: for a given
set $\mathcal{X}_r:=\{x_1\klk x_r\}\subset\fq$ of $r$ elements, the
set $\mathcal{S}_{\mathcal{X}_r}^{\bfs{a}}$ is not empty if and only
if $(x_1\klk x_r)$ is a common zero of $R_{d-s}^{\bfs{a}}\klk
R_{r-1}^{\bfs{a}}$.

The main purpose of this section is to show how the polynomials
$R_j^{\bfs{a}}$ can be expressed in terms of the elementary
symmetric polynomials $\Pi_1 ,\ldots, \Pi_s$. In order to do this,
we first obtain a recursive expression for the remainder of the
division of $T^j$ by $Q:=(T-X_1)\cdots (T-X_r)$ for $r\le j\le d$.
\begin{lemma}\label{lemma: formula Hij}
For $r\le j\le d$, the following congruence relation holds:
\begin{equation}
\label{eq: reduced expression T j} T^j\equiv
H_{r-1,j}T^{r-1}+H_{r-2,j}T^{r-2}+\cdots+H_{0,j}\mod Q,
\end{equation}
where each $H_{i,j}$ is equal to zero or an homogeneous element of
$\fq[X_1\klk X_r]$ of degree $j-i$. Furthermore, for $j-i\le r$, the
polynomial $H_{i,j}$ is a monic element of
$\fq[\Pi_1\klk\Pi_{j-i-1}][\Pi_{j-i}]$, up to a nonzero constant of
$\fq$.
\end{lemma}
\begin{proof}
We argue by induction on $j\ge r$. Taking into account that
\begin{equation}\label{eq: reduced expression T r+1}
T^r\equiv  \Pi_1T^{r-1}-\Pi_2T^{r-2}+\cdots+(-1)^{r-1}\Pi_r
\mod{Q},\end{equation}
we immediately deduce (\ref{eq: reduced expression T j}) for $j=r$.

Next assume that (\ref{eq: reduced expression T j}) holds for a
given $j$ with $r\le j$. Multiplying both sides of (\ref{eq: reduced
expression T j}) by $T$ and combining with (\ref{eq: reduced
expression T r+1}) we obtain:
\begin{eqnarray*}
T^{j+1}&\equiv& H_{r-1,j}T^r+H_{r-2,j}T^{r-1}+
\cdots+H_{0,j}T\\
&\equiv&(\Pi_1H_{r-1,j}+
H_{r-2,j})T^{r-1}+\cdots+((-1)^{r-2}\Pi_{r-1}H_{r-1,j}+H_{0,j})T\\
&&+(-1)^{r-1}\Pi_rH_{r-1,j},
\end{eqnarray*}
where both congruences are taken modulo $Q$.

Define
\begin{eqnarray*}
H_{k,j+1}&:=&(-1)^{r-1-k}\Pi_{r-k}H_{r-1,j}+H_{k-1,j}\ \textit{for}\
1\le k\le r-1,\\H_{0,j+1}&:=&(-1)^{r-1}\Pi_rH_{r-1,j}.
\end{eqnarray*} Then we have
$$T^{j+1}\equiv H_{r-1,j+1}T^{r-1}+H_{r-2,j+1}T^{r-2}+\cdots+H_{0,j+1}
\mod{Q}.$$
There remains to prove that the polynomials $H_{k,j+1}$ have the
form asserted.

Fix $k$ with $1\le k\le r-1$. Then
$H_{k,j+1}=(-1)^{r-1-k}\Pi_{r-k}H_{r-1,j}+H_{k-1,j}$. By the
inductive hypothesis we have that $H_{r-1,j}$ and $H_{k-1,j}$ are
equal to zero or homogeneous polynomials of degree $j-r+1$ and
$j-k+1$ respectively. We easily conclude that $H_{k,j+1}$ is equal
to zero or homogeneous of degree $j-k+1$. Further, for $j+1-k\le r$,
since $\max\{r-k,j-r+1\}\le j-k<r$ we see that $\Pi_{r-k}H_{r-1,j}$
is an element of the polynomial ring $\fq[\Pi_1\klk\Pi_{j-k}]$. On
the other hand, $H_{k-1,j}$ is a monic element of
$\fq[\Pi_1\klk\Pi_{j-k}][\Pi_{j-k+1}]$, up to a nonzero constant of
$\fq$, which implies that so is $H_{k,j+1}$.

Finally, for $k=0$ we have $H_{0,j+1}:=(-1)^{r-1}\Pi_rH_{r-1,j}$,
which shows that $H_{0,j+1}$ is equal to zero or an homogeneous
polynomials of $\fq[X_1\klk X_r]$ of degree $r+j-r+1=j+1$. This
finishes the proof of the lemma.
\end{proof}

We observe that an explicit expression of the polynomials $H_{i,j}$
can be obtained following the approach of \cite[Proposition
2.2]{CaMaPr12}. As we do not need such an explicit expression we
shall not pursue this point any further.

Finally we obtain an expression of the polynomials $R_j^{\bfs{a}}\in
\fq[X_1,\ldots, X_r]$ ($d-s\le j\le r-1$) in terms of the
polynomials $H_{i,j}$.
\begin{proposition}  \label{prop: formula for Rj}
Let $s,d\in\N$ with $1\le s\le d-2$ and $2(s+1)\le d$. For $d-s\le
j\le r-1$, the following identity holds:
\begin{equation} \label{eq: expression for Rj}
R_j^{\bfs{a}}=a_j+\sum_{i=r}^da_iH_{j,i},
\end{equation}
where the polynomials $H_{j,i}$ are defined in Lemma \ref{lemma:
formula Hij}. In particular, $R_j^{\bfs{a}}$ is a monic element of
$\fq[\Pi_1\klk \Pi_{d-1-j}][\Pi_{d-j}]$ of degree $d-j\le s$ for
$d-s\le j\le r-1$.
\end{proposition}
\begin{proof}
By Lemma \ref{lemma: formula Hij} we have the following congruence
relation for $r\le j\le d$:
$$T^j\equiv
H_{r-1,j}T^{r-1}+H_{r-2,j}T^{r-2}+\cdots+H_{0,j} \mod{Q}.$$
Hence we obtain
\begin{eqnarray*}
\sum_{j=d-s}^da_jT^j&=&\sum_{j=d-s}^{r-1}a_jT^j+\sum_{j=r}^da_jT^j\\
&\equiv& \sum_{j=d-s}^{r-1}a_jT^j+\sum_{j=r}^da_j
\sum_{i=d-s}^{r-1}H_{i,j}T^i+\mathcal{O}(T^{d-s-1}) \mod{Q}\\
&\equiv& \sum_{j=d-s}^{r-1}\bigg(a_j+
\sum_{i=r}^da_iH_{j,i}\bigg)T^j+\mathcal{O}(T^{d-s-1}) \mod{Q},
\end{eqnarray*}
where $\mathcal{O}(T^{d-s-1})$ represents a sum of terms of
$\fq[X_1\klk X_r][T]$ of degree at most $d-s-1$ in $T$. This shows
that the polynomials $R_j^{\bfs{a}}$ have the form asserted in
(\ref{eq: expression for Rj}). Furthermore, we observe that, for
each $H_{j,i}$ occurring in (\ref{eq: expression for Rj}), we have
$i-j\le s\le d-s-2\le r$. This implies that each $H_{j,i}$ is a
monic element of $\fq[\Pi_1\klk \Pi_{i-j-1}][\Pi_{i-j}]$ of degree
$i-j$. As a consequence, we see that $R_j^{\bfs{a}}$ is a monic
element of $\fq[\Pi_1\klk \Pi_{d-1-j}][\Pi_{d-j}]$ of degree $d-j$
for $d-s\le j\le r-1$. This finishes the proof.\end{proof}
%
%
\section{The geometry of the set of zeros of $R_{d-s}^{\bfs{a}}\klk R_{r-1}^{\bfs{a}}$}
\label{section: geometry of Vr}
For positive integers $s$, $d$ with $q<d$, $1\le s\le d-2$ and
$2(s+1)\le d$, we fix as in the previous section an $s$--tuple
$\boldsymbol{a}:=(a_{d-1}\klk a_{d-s})\in\fq^s$ and consider the
polynomial $f_{\bfs{a}}:=T^d+a_{d-1}T^{d-1}\plp a_{d-s}T^{d-s}$. For
fixed $r$ with $d-s+1\le r\le d$, in Section \ref{subsec: relating
interpol to rat points} we associate to $f_{\bfs{a}}$ polynomials
$R_j^{\bfs{a}}\in \fq[X_1,\ldots, X_r]$ ($d-s\le j\le r-1$), whose
sets of common $q$--rational zeros are relevant for our purposes.

According to Proposition \ref{prop: formula for Rj}, we may express
each $R_j^{\bfs{a}}$ as a polynomial in the first $s$ elementary
symmetric polynomials $\Pi_1,\ldots,\Pi_s$ of $\fq[X_1,\ldots,X_r]$.
More precisely, let $Y_1\klk Y_s$ be new indeterminates over $\cfq$.
Then we have that
$$R_j^{\bfs{a}}=S_j^{\bfs{a}}(\Pi_1,\ldots,\Pi_{d-j})\quad (d-s\le j\le r-1),$$
where each $S_j^{\bfs{a}}\in \fq[Y_1,\ldots,Y_{d-j}]$ is a monic
element of $\fq[Y_1,\ldots,Y_{d-1-j}][Y_{d-j}]$ of degree $1$ in
$Y_{d-j}$.

In this section we obtain critical information on the geometry of
the set of common zeros of the polynomials $R_j^{\bfs{a}}$ that will
allow us to establish estimates on the number of common
$q$--rational zeros of $R_{d-s}^{\bfs{a}}\klk R_{r-1}^{\bfs{a}}$.
%
%
\subsection{Notions of algebraic geometry}\label{subsec: alg geom}
Since our approach relies on tools of algebraic geometry, we briefly
collect the basic definitions and facts that we need in the sequel.
We use standard notions and notations of algebraic geometry, which
can be found in, e.g., \cite{Kunz85}, \cite{Shafarevich94}.

We denote by $\A^n$ the affine $n$--dimensional space $\cfq{\!}^{n}$
and by $\Pp^n$ the projective $n$--dimensional space over
$\cfq{\!}^{n+1}$. Both spaces are endowed with their respective
Zariski topologies, for which a closed set is the zero locus of
polynomials of $\cfq[X_1,\ldots, X_{n}]$ or of homogeneous
polynomials of  $\cfq[X_0,\ldots, X_{n}]$. For $\K:=\fq$ or
$\K:=\cfq$, we say that a subset $V\subset \A^n$ is an {\sf affine
$\K$--variety} if it is the set of common zeros in $\A^n$ of
polynomials $F_1,\ldots, F_{m} \in \K[X_1,\ldots, X_{n}]$.
Correspondingly, a {\sf projective $\K$--variety} is the set of
common zeros in $\Pp^n$ of a family of homogeneous polynomials
$F_1,\ldots, F_m \in\K[X_0 ,\ldots, X_n]$. We shall frequently
denote by $V(F_1\klk F_m)$ the affine or projective $\K$--variety
consisting of the common zeros of polynomials $F_1\klk F_m$. The set
$V(\fq):=V\cap \fq^n$ is the set of {\sf $q$--rational points} of
$V$.

A $\K$--variety $V$ is $\K$--{\sf irreducible} if it cannot be
expressed as a finite union of proper $\K$--subvarieties of $V$.
Further, $V$ is {\sf absolutely irreducible} if it is irreducible as
a $\cfq$--variety. Any $\K$--variety $V$ can be expressed as an
irredundant union $V=\mathcal{C}_1\cup \cdots\cup\mathcal{C}_s$ of
irreducible (absolutely irreducible) $\K$--varieties, unique up to
reordering, which are called the {\sf irreducible} ({\sf absolutely
irreducible}) $\K$--{\sf components} of $V$.


For a $\K$-variety $V$ contained in $\A^n$ or $\Pp^n$, we denote by
$I(V)$ its {\sf defining ideal}, namely the set of polynomials of
$\K[X_1,\ldots, X_n]$, or of $\K[X_0,\ldots, X_n]$, vanishing on
$V$. The {\sf coordinate ring} $\K[V]$ of $V$ is defined as the
quotient ring $\K[X_1,\ldots,X_n]/I(V)$ or
$\K[X_0,\ldots,X_n]/I(V)$. The {\sf dimension} $\dim V$ of a
$\K$-variety $V$ is the length $r$ of the longest chain
$V_0\varsubsetneq V_1 \varsubsetneq\cdots \varsubsetneq V_r$ of
nonempty irreducible $\K$-varieties contained in $V$. A
$\K$--variety is called {\sf equidimensional} if all its irreducible
$\K$--components are of the same dimension.

The {\sf degree} $\deg V$ of an irreducible $\K$-variety $V$ is the
maximum number of points lying in the intersection of $V$ with a
generic linear space $L$ of codimension $\dim V$, for which $V\cap
L$ is a finite set. More generally, following \cite{Heintz83} (see
also \cite{Fulton84}), if $V=\mathcal{C}_1\cup\cdots\cup
\mathcal{C}_s$ is the decomposition of $V$ into irreducible
$\K$--components, we define the degree of $V$ as
$$\deg V:=\sum_{i=1}^s\deg \mathcal{C}_i.$$
An important tool for our estimates is the following {\em B\'ezout
inequality} (see \cite{Heintz83}, \cite{Fulton84}, \cite{Vogel84}):
if $V$ and $W$ are $\K$--varieties, then the following inequality
holds:
\begin{equation}\label{equation:Bezout}
\deg (V\cap W)\le \deg V \cdot \deg W.
\end{equation}

We shall also make use of the following well--known identities
relating the degree of an affine $\K$--variety $V \subset \A^n$, the
degree of its projective closure (with respect to the projective
Zariski $\K$--topology) $\overline{V} \subset \Pp^n$ and the degree
of the affine cone $\widetilde{V}$ of $\overline{V}$ (see, e.g.,
\cite[Proposition 1.11]{CaGaHe91}):
%
$$\deg V= \deg \overline{V}= \deg\tilde{V}.$$

Elements $F_1 \klk F_{n-r}$ in $\K[X_1\klk X_n]$ or in $\K[X_0\klk
X_n]$ form a {\sf regular sequence} if $F_1$ is nonzero and each
$F_i$ is not a zero divisor in the quotient ring $\K[X_1\klk
X_n]/(F_1\klk F_{i-1})$ or $\K[X_0\klk X_n]/(F_1\klk F_{i-1})$ for
$2\le i\le n-r$. In such a case, the (affine or projective)
$\K$--variety $V:=V(F_1\klk F_{n-r})$ they define is equidimensional
of dimension $r$, and is called a {\sf set--theoretic} {\sf complete
intersection}. If the ideal $(F_1\klk F_{n-r})$ generated by
$F_1\klk F_{n-r}$ is radical, then we say that $V$ is an {\sf
ideal--theoretic} {\sf complete intersection}. If $V\subset\Pp^n$ is
an ideal--theoretic complete intersection defined over $\K$, of
dimension $r$ and degree $\delta$, and $F_1 \klk F_{n-r}$ is a
system of generators of $I(V)$, the degrees $d_1\klk d_{n-r}$ depend
only on $V$ and not on the system of generators. Arranging the $d_i$
in such a way that $d_1\geq d_2 \geq \cdots \geq d_{n-r}$, we call
$\boldsymbol{d}:=(d_1\klk d_{n-r})$ the {\sf multidegree} of $V$. In
particular, it follows that $\delta= \prod_{i=1}^{n-r}d_i$ holds.

Let $V$ be a variety contained in $\A^n$ and let $I(V)\subset
\cfq[X_1,\ldots, X_n]$ be the defining ideal of $V$. Let $\bfs{x}$
be a point of $V$. The {\sf dimension} $\dim_{\bfs{x}}V$ {\sf of}
$V$ {\sf at} $\bfs{x}$ is the maximum of the dimensions of the
irreducible components of $V$ that contain $\bfs{x}$. If
$I(V)=(F_1,\ldots, F_m)$, the {\sf tangent space}
$\mathcal{T}_{\bfs{x}}V$ to $V$ at $\bfs{x}$ is the kernel of the
Jacobian matrix $(\partial F_i/\partial X_j)_{1\le i\le m,1\le j\le
n}(\bfs{x})$ of the polynomials $F_1,\ldots, F_m$ with respect to
$X_1,\ldots, X_n$ at $\bfs{x}$. The point $\bfs{x}$ is {\sf regular}
if $\dim\mathcal{T}_{\bfs{x}}V=\dim_{\bfs{x}}V$ holds. Otherwise,
the point $\bfs{x}$ is called {\sf singular}. The set of singular
points of $V$ is the {\sf singular locus} $\mathrm{Sing}(V)$ of $V$.
A variety is called {\sf nonsingular} if its singular locus is
empty. For a projective variety, the concepts of tangent space,
regular and singular point can be defined by considering an affine
neighborhood of the point under consideration.

Let $V$ and $W$ be irreducibles $\K$--varieties of the same
dimension and let $f:V\to W$ be a regular map for which
$\overline{f(V)}=W$ holds, where $\overline{f(V)}$ denotes the
closure of $f(V)$ with respect to the Zariski topology of $W$. Then
$f$ induces a ring extension $\K[W]\hookrightarrow \K[V]$ by
composition with $f$. We say that $f$ is a {\sf finite morphism} if
this extension is integral, namely if each element $\eta\in\K[V]$
satisfies a monic equation with coefficients in $\K[W]$. A basic
fact is that a finite morphism is necessarily closed. Another fact
concerning finite morphisms we shall use in the sequel is that the
preimage $f^{-1}(S)$ of an irreducible closed subset $S\subset W$ is
equidimensional of dimension $\dim S$.
%
%
\subsection{The singular locus of symmetric complete intersections}
\label{subsec: singular locus symmetric complete inters}
With the notations and assumptions of the beginning of Section
\ref{section: geometry of Vr}, let $V_r^{\bfs{a}}\subset \A^r$ be
the affine $\fq$--variety defined by the polynomials
$R_{d-s}^{\bfs{a}}\klk R_{r-1}^{\bfs{a}}\in\fq[X_1\klk X_r]$. In
this section we shall establish several facts concerning the
geometry of $V_r^{\bfs{a}}$. For this purpose, we consider the
somewhat more general framework that we now introduce. This will
allow us to make more transparent the facts concerning the algebraic
structure of the family of polynomials $R_{d-s}^{\bfs{a}}\klk
R_{r-1}^{\bfs{a}}$ which are important at this point.

Let $Y_1,\ldots, Y_s$ be new indeterminates over $\cfq$ and let be
given polynomials $S_j\in\fq[Y_1,\ldots,Y_s]$ for $d-s\le j\le r-1$.
Let $(\partial \bfs{S}/\partial \bfs{Y}):=(\partial S_j/\partial
Y_k)_{d-s\le j\le r-1,1\le k\le s}$ be the Jacobian matrix of
$S_{d-s}\klk S_{r-1}$ with respect to $Y_1\klk Y_s$. Our assumptions
on $s$, $d$ and $r$ imply $r-d+s\le s$ and thus, $(\partial
\bfs{S}/\partial \bfs{Y})$ has full rank if and only if
$\mathrm{rank}(\partial \bfs{S}/\partial \bfs{Y})=r-d+s$ holds.
Assume that $S_{d-s}\klk S_{r-1}$ satisfy the following conditions:
\begin{itemize}
  \item[(H1)] $S_{d-s}\klk S_{r-1}$ form a regular sequence of $\fq[Y_1\klk
  Y_s]$;
  \item[(H2)] $(\partial \bfs{S}/\partial \bfs{Y})(\bfs{y})$ has full rank $r-d+s$ for
every $\bfs{y}\in\A^s$.
\end{itemize}
From (H1) and (H2) we immediately conclude that the affine variety
$W_r\subset\A^s$ defined by $S_{d-s}\klk S_{r-1}$ is a nonsingular
set--theoretic complete intersection of dimension $d-r$.
Furthermore, as a consequence of \cite[Theorem 18.15]{Eisenbud95} we
conclude that $S_{d-s}\klk S_{r-1}$ define a radical ideal, and
hence $W_r$ is an ideal--theoretic complete intersection.

Denote by $\Pi_1,\ldots,\Pi_s$ the first $s$ elementary symmetric
polynomials of $\fq[X_1,\ldots, X_r]$ and let
$R_j:=S_j(\Pi_1,\ldots, \Pi_s)$ for $d-s\le j\le r-1$. We denote by
$V_r\subset\A^r$ the affine variety defined by $R_{d-s}\klk
R_{r-1}$. In what follows we shall establish several facts
concerning the geometry of $V_r$.

For this purpose, we consider the following surjective morphism of
$\fq$--varieties:
\begin{eqnarray*}
   {\Pi}^{\bfs{r}}: \A^r & \rightarrow & \A^r
   \\
   \bfs{x} & \mapsto & (\Pi_1(\bfs{x}),\ldots,\Pi_r(\bfs{x})).
\end{eqnarray*}
It is easy to see that $\Pi^{\bfs{r}}$ is finite morphism (see,
e.g., \cite[\S 5.3, Example 1]{Shafarevich94}). In particular, the
preimage $(\Pi^{\bfs{r}})^{-1}(Z)$ of an irreducible affine variety
$Z\subset\A^r$ of dimension $m$ is equidimensional and of dimension
$m$ (see, e.g., \cite[\S 4.2, Proposition]{Danilov94}).

We now consider $S_{d-s}\klk S_{r-1}$ as elements of $\fq[Y_1\klk
Y_r]$. Since they form a regular sequence, the affine variety
$W_j^{\bfs{r}}=V(S_{d-s}\klk S_j)\subset\A^r$ is equidimensional of
dimension $r-j+d-s-1$. This implies that the affine variety
$V_j^{\bfs{r}}=(\Pi^{\bfs{r}})^{-1}(W_j^{\bfs{r}})$ defined by
$R_{d-s}\klk R_j$ is equidimensional of dimension $r-j+d-s-1$. We
conclude that the polynomials $R_{d-s}\klk R_{r-1}$ form a regular
sequence of $\fq[X_1\klk X_r]$ and deduce the following result.
\begin{lemma}\label{lemma: Vr is complete inters}
Let $V_r\subset \A^r$ be the $\fq$--variety defined by $R_{d-s}\klk
R_{r-1}$. Then $V_r$ is a set--theoretic complete intersection of
dimension $d-s$.
\end{lemma}

Next we discuss the dimension of the singular locus of $V_r$. For
this purpose, we consider the following surjective morphism of
$\fq$--varieties:
 \begin{eqnarray*}
   \Pi : V_r & \rightarrow & W_r
   \\
   \bfs{x} & \mapsto & (\Pi_1(\bfs{x}),\ldots,\Pi_s(\bfs{x})).
 \end{eqnarray*}
For $\bfs{x}\in V_r$ and $\bfs{y}:=\Pi(\bfs{x})$, we denote by
$\mathcal{T}_{\bfs{x}}V_r$ and $\mathcal{T}_{\bfs{y}} W_r$ the
tangent spaces to $V_r$ at $\bfs{x}$ and to $W_r$ at $\bfs{y}$. We
also consider the differential map of $\Pi$ at $\bfs{x}$, namely
 \begin{eqnarray*}
   \mathrm{d}_{\bfs{x}}\Pi :  \mathcal{T}_{\bfs{x}}V_r &
   \rightarrow & \mathcal{T}_{\bfs{y}}W_r \\
   \bfs{v} & \mapsto & A(\bfs{x})\cdot \bfs{v},
 \end{eqnarray*}
where $A(\bfs{x})$ stands for the  $(s\times r)$--matrix
  \begin{equation} \label{eq: Jacobian Pi}
    A(\bfs{x}):=\left(\frac{\partial \Pi}{\partial
    \bfs{X}}\right)(\bfs{x}):=
       \left(
         \begin{array}{ccc}
           \dfrac{\partial\Pi_1}{\partial X_1}(\bfs{x}) & \cdots & \dfrac{\partial\Pi_1}{\partial X_r}(\bfs{x})
         \\
           \vdots & & \vdots
          \\
           \dfrac{\partial\Pi_s}{\partial X_1}(\bfs{x})& \cdots & \dfrac{\partial\Pi_s}{\partial X_r}(\bfs{x})
         \end{array}
       \right).
  \end{equation}
In order to prove our result about the singular locus of $V_r$, we
first make a few remarks concerning the Jacobian matrix of the
elementary  symmetric polynomials that will be useful in the sequel.

It is well known that the first partial derivatives of the
elementary symmetric polynomials $\Pi_i$  satisfy the following
equalities (see, e.g., \cite{LaPr02}) for $1\leq i,j \leq r$:
  \begin{equation}\label{eq:derivadas parciales simetricos elementales}
    \frac{\partial \Pi_i}{\partial X_{j}}= \Pi_{i-1}-X_{j} \Pi_{i-2}
    + X_{j}^2 \Pi_{i-3} +\cdots+ (-1)^{i-1} X_{j}^{i-1}.
  \end{equation}
As a consequence, denoting by $A_r$ the $(r\times r)$--Vandermonde
matrix
  \begin{equation}\label{eq:matriz de vandermonde k+1}
    A_r:=\left(
         \begin{array}{cccc}
           1 & 1 & \cdots & 1
         \\
           X_1 & X_2 & \cdots &  X_r
          \\  \vdots & \vdots & & \vdots
          \\
           X_1^{r-1} & X_2^{r-1} & \cdots & X_r^{r-1},
         \end{array}
       \right),
 \end{equation}
we deduce that the  Jacobian matrix of $\Pi_1,\dots,\Pi_r$ with
respect to $X_1,\dots,X_r$ can be factored as follows:
\begin{equation} \label{eq:factorizacion del jacobiano de los simetricos elementales}
     \left(\frac{\partial \Pi_i}{\partial X_j}\right)_{1\le i,j\le r}:=B_r\cdot
      A_r
       :=
        \left(
         \begin{array}{ccccc}
           1 & \ 0 & 0 &  \dots & 0
         \\
           \Pi_1 & - 1 & 0 &  &
          \\
           \Pi_2 & -\Pi_1 & 1 & \ddots & \vdots
          \\
           \vdots &\vdots  & \vdots & \ddots & 0
           \\
           \Pi_{r-1} & -\Pi_{r-2} & \Pi_{r-3} & \cdots &\!\! (-1)^{r-1}
         \end{array}
       \!\!\right)
     \cdot
         A_r.
  \end{equation}
We observe that the left factor $B_r$ is a square, lower--triangular
matrix whose determinant is equal to $(-1)^{{(r-1)r}/{2}}$. This
implies that the determinant of the matrix
$(\partial{\Pi_i}/\partial{X_j})_{1\leq i,j\leq r}$ is equal, up to
a sign, to the determinant of $A_r$, i.e.,
$$\det \left(\frac{\partial \Pi_i}{\partial X_j}\right)_{1\le i,j\le
r}=(-1)^{{(r-1)r}/{2}} \prod_{1\le i < j\le r}(X_i-X_j).$$

Let $(\partial \bfs{R}/\partial \bfs{X}):=(\partial R_j/\partial
X_k)_{d-s\le j\le r-1,1\leq k \le r}$ be the Jacobian matrix of the
polynomials $R_{d-s}\klk R_{r-1}$ with respect to $X_1\klk X_r$.
\begin{theorem}\label{theorem: dimension singular locus}
The set of points $\bfs{x}\in\A^r$ for which $(\partial
\bfs{R}/\partial \bfs{X})(\bfs{x})$ has not full rank has dimension
at most $s-1$. In particular, the singular locus $\Sigma_r$ of $V_r$
has dimension at most $s-1$.
\end{theorem}
\begin{proof}
By the chain rule we deduce that the partial derivatives of $R_j$
satisfy the following equality for $1\le k \le r$:
$$
\dfrac{\partial R_j}{\partial X_k}  = \bigg(\dfrac{\partial
S_j}{\partial Y_1} \circ \Pi \bigg)
\cdot\dfrac{\partial\Pi_1}{\partial X_k} + \cdots +
\bigg(\dfrac{\partial S_j}{\partial Y_s}\circ \Pi \bigg) \cdot
\dfrac{\partial\Pi_s}{\partial X_k}
$$
Therefore we obtain
$$
\left(\frac{\partial \bfs{R}}{\partial
\bfs{X}}\right)=\left(\frac{\partial \bfs{S}}{\partial
\bfs{Y}}\circ\Pi\right)\cdot \left(\frac{\partial \Pi}{\partial
\bfs{X}}\right).
$$

Fix an arbitrary point $\bfs{x}$ for which $(\partial
\bfs{R}/\partial \bfs{X})(\bfs{x})$ has not full rank. Let
$\bfs{v}\in\A^{r-d+s}$ a nonzero vector in the left kernel of
$(\partial \bfs{R}/\partial \bfs{X})(\bfs{x})$. Then
$$\mathbf{0}=
\bfs{v}\cdot \left(\frac{\partial \bfs{R}}{\partial
\bfs{X}}\right)(\bfs{x})=\bfs{v}\cdot\left(\frac{\partial
\bfs{S}}{\partial \bfs{Y}}\right)\big(\Pi(\bfs{x})\big)\cdot
A(\bfs{x}),$$
where $A(\bfs{x})$ is the matrix  defined in (\ref{eq: Jacobian
Pi}). Since by (H2) the Jacobian matrix $(\partial \bfs{S}/\partial
\bfs{Y})\big(\Pi(\bfs{x})\big)$ has full rank,
$\bfs{w}:=\bfs{v}\cdot\left({\partial \bfs{S}}/{\partial
\bfs{Y}}\right)\big(\Pi(\bfs{x})\big)\in\A^s$ is nonzero and
$$\bfs{w}\cdot A(\bfs{x})= \mathbf{0}.$$
Hence, all the maximal minors of $A(\bfs{x})$ must be zero.

The matrix $A(\bfs{x})$ is the $(s\times r)$--submatrix of
$(\partial{\Pi_i}/\partial{X_j})_{1\leq i,j\leq r}(\bfs{x})$
consisting of the first $s$ rows of the latter. Therefore, from
(\ref{eq:factorizacion del jacobiano de los simetricos elementales})
we conclude that
$$
A(\bfs{x})=B_{s,r}(\bfs{x})\cdot A_r(\bfs{x}),
$$
where $B_{s,r}(\bfs{x})$ is the $(s\times r)$--submatrix of
$B_r(\bfs{x})$ consisting of the first $s$ rows of $B_r(\bfs{x})$.
Since the last $r-s$ columns of $B_{s,r}(\bfs{x})$ are zero, we may
rewrite this identity in the following way:
\begin{equation}\label{eq:factorization submatrix Jacobian elem symm}
A(\bfs{x})=B_s(\bfs{x})\cdot \left(
         \begin{array}{cccc}
           1 & 1 & \dots & 1
         \\
           x_1 & x_2 & \dots &  x_r
          \\  \vdots & \vdots & & \vdots
          \\
           x_1^{s-1} & x_2^{s-1} & \dots & x_r^{s-1}
         \end{array}
       \right),
\end{equation}
where $B_s(\bfs{x})$ is the $(s \times s)$--submatrix of
$B_r(\bfs{x})$ consisting on the first $s$ rows and the first $s$
columns of $B_r(\bfs{x})$.

Fix $1\leq l_1<\cdots<l_s\leq r$, set $I:=(l_1,\ldots,l_s)$ and
consider the $(s\times s)$--submatrix $M_I(\bfs{x})$ of $A(\bfs{x})$
consisting of the columns $l_1,\ldots,l_s$ of $A(\bfs{x})$, namely
$M_I(\bfs{x}):=({\partial \Pi_i}/\partial X_{l_{j}})_{1\leq i,j\leq
s}(\bfs{x})$.

From (\ref{eq:factorizacion del jacobiano de los simetricos
elementales}) and (\ref{eq:factorization submatrix Jacobian elem
symm}) we easily see that $M_I(\bfs{x})=B_s(\bfs{x})\cdot
A_{s,I}(\bfs{x})$, where $A_{s,I}(\bfs{x})$ is the Vandermonde
matrix $A_{s,I}(\bfs{x}):=(x_{l_{j}}^{i-1})_{1\leq i,j\leq s}$.
Therefore, we obtain
\begin{equation}\label{eq:determinante de vandermonde}
\det\big(M_I(\bfs{x})\big) = (-1)^{\frac{(s-1)s}{2}}\det
A_{s,I}(\bfs{x}) = (-1)^{\frac{(s-1)s}{2}}\!\!\prod_{1\le m<n\le
s}\!\!(x_{l_{m}}-x_{l_{n}})=0.
\end{equation}

Since (\ref{eq:determinante de vandermonde}) holds for every
$I:=(l_1,\ldots,l_s)$ as above, we conclude that $\bfs{x}$ has at
most $s-1$ pairwise--distinct coordinates. In particular, the set of
points $\bfs{x}$ for which $\mathrm{rank}(\partial \bfs{R}/\partial
\bfs{X})(\bfs{x})<r-d+s$ is contained in a finite union of linear
varieties of $\A^r$ of dimension $s-1$, and thus is an affine
variety of dimension at most $s-1$.

Now let $\bfs{x}$ be an arbitrary point $\Sigma_r$. By Lemma
\ref{lemma: Vr is complete inters} we have $\dim
\mathcal{T}_{\bfs{x}}V_r>d-s$. This implies that
$\mathrm{rank}\left({\partial \bfs{R}}/{\partial
\bfs{X}}\right)(\bfs{x})<r-d+s$, for otherwise we would have $\dim
\mathcal{T}_{\bfs{x}}V_r\le d-s$, contradicting thus the fact that
$\bfs{x}$ is a singular point of $V_r$. This finishes the proof of
the theorem.
\end{proof}
From Lemma \ref{lemma: Vr is complete inters} and Theorem
\ref{theorem: dimension singular locus} we obtain further algebraic
and geometric consequences concerning the polynomials $R_j$ and the
variety $V_r$. By Theorem \ref{theorem: dimension singular locus} we
have that the set of points $\bfs{x}\in\A^r$ for which the Jacobian
matrix $(\partial \bfs{R}/\partial \bfs{X})(\bfs{x})$ has not full
rank has dimension at most $s-1$. Since $R_{d-s}\klk R_{r-1}$ form a
regular sequence and $s-1<d-s$ holds, from \cite[Theorem
18.15]{Eisenbud95} we conclude that $R_{d-s}\klk R_{r-1}$ define a
radical ideal of $\fq[X_1\klk X_r]$. Finally, by the B\'ezout
inequality (\ref{equation:Bezout}) we have $\deg V_r\le
\prod_{j=d-s}^{r-1}\deg R_j$. In other words, we have the following
statement.
\begin{corollary}\label{coro: radicality and degree Vr}
The polynomials $R_{d-s}\klk R_{r-1}$ define a radical ideal and the
variety $V_r$ has degree $\deg V_r\le\prod_{j=d-s}^{r-1}\deg R_j$.
\end{corollary}
%
%
\subsection{The geometry of $V_r^{\bfs{a}}$}

Now we consider the affine $\fq$--variety $V_r^{\bfs{a}}\subset\A^r$
defined by the polynomials $R_{d-s}^{\bfs{a}}\klk
R_{r-1}^{\bfs{a}}\in \fq[X_1,\ldots X_r]$ associated to
$\boldsymbol{a}:=(a_{d-1}\klk a_{d-s})\in\fq^s$ and the polynomial
$f_{\bfs{a}}:=T^d+a_{d-1}T^{d-1}\plp a_{d-s}T^{d-s}$. According to
Proposition \ref{prop: formula for Rj}, we may express each
$R_j^{\bfs{a}}$ in the form
$R_j^{\bfs{a}}=S_j^{\bfs{a}}(\Pi_1,\ldots,\Pi_{d-j})$, where
$S_j^{\bfs{a}}\in \fq[Y_1,\ldots,Y_{d-j}]$ is a monic polynomial in
$Y_{d-j}$, up to a nonzero constant, of degree $1$ in $Y_{d-j}$. In
particular, by a recursive argument it is easy to see that
$$\cfq[Y_1\klk Y_s]/(S_{d-s}^{\bfs{a}}\klk S_j^{\bfs{a}})\simeq\cfq[Y_1\klk Y_{d-j-1}]$$
for $d-s\le j\le r-1$. We conclude that $S_{d-s}^{\bfs{a}}\klk
S_{r-1}^{\bfs{a}}$ form a regular sequence of $\fq[Y_1\klk Y_s]$,
namely they satisfy (H1). Furthermore, we observe that
$$
\left(\frac{\partial \bfs{S}^{\bfs{a}}}{\partial
\bfs{Y}}\right)(\bfs{y})=
\left(\begin{array}{cccccccc}\dfrac{\partial
S_{d-s}^{\bfs{a}}}{\partial Y_1}(\bfs{y})
  &\cdots& \dfrac{\partial S_{d-s}^{\bfs{a}}}{\partial
  Y_{d-r}}(\bfs{y})&&\cdots& &\!\!\!\!\!\!\!\!c_{d-s}
  \\[2ex]
\dfrac{\partial S_{d-s+1}^{\bfs{a}}}{\partial Y_1}(\bfs{y})
  &\cdots& \dfrac{\partial S_{d-s+1}^{\bfs{a}}}{\partial
  Y_{d-r}}(\bfs{y})&&\cdots&c_{d-s+1}
  \\\vdots&&\vdots&&{}_{\displaystyle.}\!\cdot\!{}^{\displaystyle \cdot}&\\
  \dfrac{\partial S_{r-1}^{\bfs{a}}}{\partial
Y_1}(\bfs{y})
  &\cdots& \dfrac{\partial S_{r-1}^{\bfs{a}}}{\partial
  Y_{d-r}}(\bfs{y})&c_{r-1}

  \end{array}\right)
$$
holds for every $\bfs{y}\in \mathbb{A}^s$, where $c_{d-s}\klk
c_{r-1}$ are certain nonzero elements of $\fq$. As a consequence, we
have that  $(\partial \bfs{S}^{\bfs{a}}/\partial \bfs{Y})(\bfs{y})$
has full rank for every $\bfs{y}\in\A^s$, that is,
$S_{d-s}^{\bfs{a}}\klk S_{r-1}^{\bfs{a}}$ satisfy (H2). Then the
results of Section \ref{subsec: singular locus symmetric complete
inters} can be applied to $V_r^{\bfs{a}}$. In particular, we have
the following immediate consequence of Lemma \ref{lemma: Vr is
complete inters}, Theorem \ref{theorem: dimension singular locus}
and Corollary \ref{coro: radicality and degree Vr}.
\begin{corollary}\label{coro:dimension singular locus Vf}
Let $V_r^{\bfs{a}}\subset \A^r$ be the $\fq$--variety defined by
$R_{d-s}^{\bfs{a}}\klk R_{r-1}^{\bfs{a}}$. Then $V_r^{\bfs{a}}$ is
an ideal--theoretic complete intersection of dimension $d-s$, degree
at most $s!/(d-r)!$ and singular locus $\Sigma_r^{\bfs{a}}$ of
dimension at most $s-1$.
\end{corollary}
%
%
\subsubsection{The projective closure of $V_r^{\bfs{a}}$}
In order to obtain estimates on the number of $q$--rational points
of $V_r^{\bfs{a}}$ we also need information concerning the behavior
of $V_r^{\bfs{a}}$ ``at infinity''. For this purpose, we consider
the projective closure
$\mathrm{pcl}(V_r^{\bfs{a}})\subset\mathbb{P}^r$ of $V_r^{\bfs{a}}$,
whose definition we now recall. Consider the embedding of $\A^r$
into the projective space $\Pp^r$ which assigns to any
$\bfs{x}:=(x_1,\ldots, x_r)\in\A^r$ the point
$(1:x_1:\dots:x_r)\in\Pp^r$. The closure
$\mathrm{pcl}(V_r^{\bfs{a}})\subset\Pp^r$ of the image of
$V_r^{\bfs{a}}$ under this embedding in the Zariski topology of
$\Pp^r$ is called the {\sf projective closure} of $V_r^{\bfs{a}}$.
The points of $\mathrm{pcl}(V_r^{\bfs{a}})$ lying in the hyperplane
$\{X_0=0\}$ are called the points of $\mathrm{pcl}(V_r^{\bfs{a}})$
at infinity.

It is well--known that $\mathrm{pcl} (V_r^{\bfs{a}})$ is the
$\fq$--variety of $\mathbb{P}^r$ defined by the homogenization
$F^h\in\fq[X_0,\ldots,X_r]$ of each polynomial $F$ belonging to the
ideal $(R_{d-s}^{\bfs{a}}\klk
R_{r-1}^{\bfs{a}})\subset\fq[X_1,\ldots,X_r]$ (see, e.g., \cite[\S
I.5, Exercise 6]{Kunz85}). Denote by $(R_{d-s}^{\bfs{a}}\klk
R_{r-1}^{\bfs{a}})^h$ the ideal generated by all the polynomials
$F^h$ with $F\in (R_{d-s}^{\bfs{a}}\klk R_{r-1}^{\bfs{a}})$. Since
$(R_{d-s}^{\bfs{a}}\klk R_{r-1}^{\bfs{a}})$ is radical it turns out
that $(R_{d-s}^{\bfs{a}}\klk R_{r-1}^{\bfs{a}})^h$ is also a radical
ideal (see, e.g., \cite[\S I.5, Exercise 6]{Kunz85}). Furthermore,
$\mathrm{pcl} (V_r^{\bfs{a}})$ is an equidimensional variety of
dimension $d-s$ (see, e.g., \cite[Propositions I.5.17 and
II.4.1]{Kunz85}) and degree at most $s!/(d-r)!$ (see, e.g.,
\cite[Proposition 1.11]{CaGaHe91}).

Now we discuss the behavior of $\mathrm{pcl} (V_r^{\bfs{a}})$ at
infinity. By Proposition \ref{prop: formula for Rj}, for $d-s\le
j\le r-1$ we have
$$R_j^{\bfs{a}}=a_j+\sum_{i=r}^da_iH_{j,i},$$
where the polynomials $H_{j,i}$ are homogeneous of degree $i-j$.
Hence, the homogenization of each $R_j^{\bfs{a}}$ is the following
polynomial of $\fq[X_0,\ldots, X_r]$:
\begin{equation} \label{eq:homogenizacion de H_d}
R_j^{\bfs{a},h}=a_jX_0^{d-j}+\sum_{i=r}^da_iH_{j,i}X_0^{d-i}.
\end{equation}
In particular, it follows that $R_j^{\bfs{a},h}(0,X_1\klk
X_r)=H_{j,d}$ ($d-s\le j\le r-1$) are the polynomials associated to
the polynomial $T^d\in\fq[T]$ in the sense of Lemma \ref{lemma:
system defining Vr}.
\begin{lemma}\label{lemma: dim singular locus Vfr at infinity}
$\mathrm{pcl}(V_r^{\bfs{a}})$ has singular locus at infinity of
dimension at most $s-2$.
\end{lemma}
\begin{proof}
Let $\Sigma^{\bfs{a}}_{r,\infty}\subset\mathbb{P}^r$ denote the
singular locus of $\mathrm{pcl}(V_r^{\bfs{a}})$ at infinity, namely
the set of singular points of $\mathrm{pcl}(V_r^{\bfs{a}})$ lying in
the hyperplane $\{X_0=0\}$, and let $\bfs{x}:=(0:x_1:\dots: x_r)$ be
an arbitrary point of $\Sigma_{r,\infty}^{\bfs{a}}$. Since the
polynomials $R_j^{{\bfs{a}},h}$ vanish identically in
$\mathrm{pcl}(V_r^{\bfs{a}})$, we have
$R_j^{{\bfs{a}},h}(\bfs{x})=H_{j,d}(x_1\klk x_r)=0$ for $d-s\le j\le
r-1$. Let $(\partial H_{d}/\partial \bfs{X})$ be the Jacobian matrix
of $\{H_{j,d}:d-s\le j\le r-1\}$ with respect to $X_1\klk X_r$. We
have
\begin{equation}\label{eq: rank singular points at infinity}
\mathrm{rank}\left(\frac{\partial H_d}{\partial
\bfs{X}}\right)(\bfs{x})<r-d+s,
\end{equation}
for if not, we would have that $\dim
\mathcal{T}_{\bfs{x}}(\mathrm{pcl}(V_r^{\bfs{a}}))\le d-s$, which
implies that $\bfs{x}$ is a nonsingular point of
$\mathrm{pcl}(V_r^{\bfs{a}})$, contradicting thus the hypothesis on
$\bfs{x}$.

From Proposition \ref{prop: formula for Rj} it follows that the
polynomials $H_{j,d}$ ($d-s\le j\le r-1$) satisfy the hypotheses of
Theorem \ref{theorem: dimension singular locus}. Then Theorem
\ref{theorem: dimension singular locus} shows that the set of points
satisfying (\ref{eq: rank singular points at infinity}) is an affine
equidimensional cone of dimension at most $s-1$. We conclude that
the projective variety $\Sigma_{r,\infty}^{\bfs{a}}$ has dimension
at most $s-2$.
\end{proof}
\begin{theorem}\label{theorem: Vf at infinity}
$\mathrm{pcl}(V_r^{\bfs{a}})\cap\{X_0=0\}\subset\Pp^{r-1}$ is an
absolutely irreducible ideal--theoretic complete intersection of
dimension $d-s-1$, degree $s!/(d-r)!$, and singular locus of
dimension at most $s-2$.
\end{theorem}
\begin{proof}
From (\ref{eq:homogenizacion de H_d}) it is easy to see that the
polynomials $H_{j,d}$ vanish identically in
$\mathrm{pcl}(V_r^{\bfs{a}})\cap\{X_0=0\}$ for $d-s\le j\le r-1$.
Lemma \ref{lemma: formula Hij} shows that $\{H_{j,d}:d-s\le j\le
r-1\}$ satisfy the conditions (H1) and (H2). Then Corollary
\ref{coro:dimension singular locus Vf} shows that the variety of
$\A^r$ defined by $H_{j,d}$ ($d-s\le j\le r-1$) is an affine
equidimensional cone of dimension $d-s$, degree at most $s!/(d-r)!$
and singular locus of dimension at most $s-1$. It follows that the
projective variety of $\Pp^{r-1}$ defined by these polynomials is
equidimensional of dimension $d-s-1$, degree at most $s!/(d-r)!$ and
singular locus of dimension at most $s-2$.

Observe that $V(H_{j,d}:d-s\le j\le r-1)\subset\Pp^{r-1}$ is a
set--theoretic complete intersection, whose singular locus has
codimension at least $d-s-1-(s-2)\ge 3$. Therefore, the Hartshorne
connectedness theorem (see, e.g., \cite[Theorem 4.2]{Kunz85}) shows
that $V(H_{j,d}:d-s\le j\le r-1)$ is absolutely irreducible.

On the other hand, since $\mathrm{pcl}(V_r^{\bfs{a}})$ is
equidimensional of dimension $d-s$ we have that each irreducible
component of $\mathrm{pcl}(V_r^{\bfs{a}})\cap\{X_0=0\}$ has
dimension at least $d-s-1$. Furthermore,
$\mathrm{pcl}(V_r^{\bfs{a}})\cap\{X_0=0\}$ is contained in the
projective variety $V(H_{j,d}:d-s\le j\le r-1)$, which is absolutely
irreducible of dimension $d-s-1$. We conclude that
$\mathrm{pcl}(V_r^{\bfs{a}})\cap\{X_0=0\}$ is also absolutely
irreducible of dimension $d-s-1$, and hence
$$\mathrm{pcl}(V_r^{\bfs{a}})\cap\{X_0=0\}=V(H_{j,d}:d-s\le j\le
r-1).$$

Finally, by \cite[Theorem 18.15]{Eisenbud95} we deduce that the
polynomials $H_{j,d}$ $(d-s\le j\le r-1)$ define a radical ideal. As
a consequence, we conclude that
$\deg(\mathrm{pcl}\big(V_r^{\bfs{a}})\cap\{X_0=0\}\big)=
\prod_{j=d-s}^{r-1}\deg H_{j,d}=s!/(d-r)!$ (see, e.g., \cite[Theorem
18.3]{Harris92}). This finishes the proof of the theorem.
\end{proof}

We conclude this section with a statement that summarizes all the
facts we shall need concerning the projective closure
$\mathrm{pcl}(V_r^{\bfs{a}})$.
\begin{theorem}\label{theorem: proj closure of Vf is abs irred}
The projective variety $\mathrm{pcl}(V_r^{\bfs{a}})\subset\Pp^r$ is
an absolutely irreducible ideal--theoretic complete intersection of
dimension $d-s$, degree $s!/(d-r)!$ and singular locus of dimension
at most $s-1$.
\end{theorem}
\begin{proof}
We have already shown that $\mathrm{pcl}(V_r^{\bfs{a}})$ is an
equidimensional variety of dimension $d-s$ and degree at most
$s!/(d-r)!$. According to Corollary \ref{coro:dimension singular
locus Vf}, the singular locus of $\mathrm{pcl}(V_r^{\bfs{a}})$ lying
in the open set $\{X_0\not=0\}$ has dimension at most $s-1$, while
Lemma \ref{lemma: dim singular locus Vfr at infinity} shows that the
singular locus at infinity has dimension at most $s-2$. This shows
that the singular locus of $\mathrm{pcl}(V_r^{\bfs{a}})$ has
dimension at most $s-1$.

On the other hand, we observe that $\mathrm{pcl}(V_r^{\bfs{a}})$ is
contained in the projective variety $V(R_j^{{\bfs{a}},h}:d-s\le j\le
r-1)$. We have the inclusions
\begin{eqnarray*}
V(R_j^{{\bfs{a}},h}:d-s\le j\le r-1)\cap\{X_0\not=0\}&\!\!\!\!
\subset&
\!\!\!\!V(R_j^{\bfs{a}}:d-s\le j\le r-1)\\
V(R_j^{{\bfs{a}},h}:d-s\le j\le r-1)\cap\{X_0=0\}&\!\!\!\!\subset&
\!\!\!\!V(H_{d,j}:d-s\le j\le r-1).
\end{eqnarray*}
Both $\{R_j^{\bfs{a}}:d-s\le j\le r-1\}$ and $\{H_{j,d}:d-s\le j\le
r-1\}$ satisfy the conditions (H1) and (H2). Then Corollary
\ref{coro:dimension singular locus Vf} shows that
$V(R_j^{\bfs{a}}:d-s\le j\le r-1)\subset\A^r$ is equidimensional of
dimension $d-s$ and $V(H_{d,j}:d-s\le j\le r-1)\subset\Pp^{r-1}$ is
equidimensional of dimension $d-s-1$. We conclude that
$V(R_j^{{\bfs{a}},h}:d-s\le j\le r-1)$ has dimension at most $d-s$.
Taking into account that it is defined by $r-d+s$ polynomials, we
deduce that it is a set--theoretic complete intersection of
dimension $r-(r-d+s)=d-s$. Finally, since its singular locus has
dimension at most $s-1$ and $d-s-(s-1)\ge 3$, the Hartshorne
connectedness theorem (see, e.g., \cite[Theorem 4.2]{Kunz85}) proves
that $V(R_j^{{\bfs{a}},h}:d-s\le j\le r-1)$ is absolutely
irreducible

Summarizing, we have that $\mathrm{pcl}(V_r^{\bfs{a}})$ and
$V(R_j^{{\bfs{a}},h}:d-s\le j\le r-1)$ are projective
equidimensional varieties of dimension $d-s$ with
$\mathrm{pcl}(V_r^{\bfs{a}})\subset V(R_j^{{\bfs{a}},h}:d-s\le j\le
r-1)$ and $V(R_j^{{\bfs{a}},h}:d-s\le j\le r-1)$ absolutely
irreducible. Therefore, we deduce that
$$
\mathrm{pcl}(V_r^{\bfs{a}})=V(R_j^{{\bfs{a}},h}:d-s\le j\le r-1).
$$
Now we argue as in the proof of Theorem \ref{theorem: Vf at
infinity}. By \cite[Theorem 18.15]{Eisenbud95} it follows that the
polynomials $R_j^{\bfs{a},h}$ $(d-s\le j\le r-1)$ define a radical
ideal. This in turn implies that
$\deg\mathrm{pcl}\big(V_r^{\bfs{a}})=\prod_{j=d-s}^{r-1}\deg
R_j^{\bfs{a},h}=s!/(d-r)!$ (see, e.g., \cite[Theorem
18.3]{Harris92}) and finishes the proof of the theorem.
\end{proof}
%
%
\section{The number of $q$--rational points of $V_r^{\bfs{a}}$}
As before, let be given integers $d$ and $s$ with $d<q$, $1\le s\le
d-2$ and $2(s+1)\le d$. Let also be given $\bfs{a}:=(a_{d-1}\klk
a_{d-s})$ and set $f_{\bfs{a}}:=T^d+a_{d-1}T^{d-1}+\cdots+
a_{d-s}T^{d-s}\in\fq[T]$. As asserted before, our objective is to
determine the asymptotic behavior of the average value set
$\mathcal{V}(d,s,\bfs{a})$ of (\ref{eq: average value set}).

For this purpose, according to Theorem \ref{theorem: interpolation
problem - reduction to interp sets}, we have to determine, for
$d-s+1\le r\le d$, the number $\chi_r^{\bfs{a}}$ of subsets
$\mathcal{X}_r\subset\fq$ of $r$ elements such that there exists
$g\in\fq[T]$ of degree at most $d-s-1$ interpolating $-f_{\bfs{a}}$
at all the elements of $\mathcal{X}_r$. In Section \ref{subsec:
relating interpol to rat points} we associate to ${\bfs{a}}$ certain
polynomials $R_j^{\bfs{a}}\in\fq[X_1\klk X_r]$ $(d-s\le j\le r-1)$
with the property that the number of common $q$--rational zeros of
$R_{d-s}^{\bfs{a}}\klk R_{r-1}^{\bfs{a}}$ with pairwise distinct
coordinates equals $r!\chi_r^{\bfs{a}}$, namely
$$\chi_r^{\bfs{a}}=\frac{1}{r!}\left|\left\{\bfs{x}\in\fq^r:R_j^{\bfs{a}}(\bfs{x})=0\,
(d-s\le j\le r-1),x_k\not= x_l\, (1\le k<l\le r)\right\}\right|.$$

The results of Section \ref{section: geometry of Vr} are fundamental
for establishing the asymptotic behavior of $\chi_r^{\bfs{a}}$. Fix
$r$ with $d-s+1\le r\le d$, let $V_r^{\bfs{a}}\subset\A^r$ be the
affine variety defined by $R_{d-s}^{\bfs{a}}\klk
R_{r-1}^{\bfs{a}}\in\fq[X_1,\ldots X_r]$ and denote by
$\mathrm{pcl}(V_r^{\bfs{a}})\subset\Pp^r$ the projective closure of
$V_r^{\bfs{a}}$. According to Theorems \ref{theorem: Vf at infinity}
and \ref{theorem: proj closure of Vf is abs irred}, both
$\mathrm{pcl}(V_r^{\bfs{a}})\cap\{X_0=0\}\subset\Pp^{r-1}$ and
$\mathrm{pcl}(V_r^{\bfs{a}})\subset\Pp^r$ are projective, absolutely
irreducible, ideal--theoretic complete intersections defined over
$\fq$, of dimension $d-s-1$ and $d-s$ respectively, both of degree
$s!/(d-r)!$, having a singular locus of dimension at most $s-2$ and
$s-1$ respectively.
%
%
\subsection{Estimates on the number of $q$--rational points of
complete intersections}
In what follows, we shall use an estimate on the number of
$q$--rational points of a projective complete intersection defined
over $\fq$ due to \cite{CaMaPr13} (see \cite{GhLa02a}, \cite{GhLa02}
for further explicit estimates of this type). In \cite[Corollary
8.4]{CaMaPr13} the authors prove that, for an absolutely irreducible
ideal--theoretic complete intersection $V\subset\Pp^m$ of dimension
$n:=m-r$, degree $\delta\ge 2$, which is defined over $\fq$ by
polynomials of degree $d_1\ge\cdots\ge d_r\ge 2$, and having
singular locus of dimension at most $s\le n-3$, the number
$|V(\fq)|$ of $q$--rational points of $V$ satisfies the estimate
  \begin{equation}\label{eq: estimate rat points CML}
    \big||V(\fq)|-p_n\big|\leq 14 D^3\delta^2q^{n-1},
  \end{equation}
where $p_n:=q^n+q^{n-1}+\cdots+ q+1$ is the cardinality of
$\Pp^n(\fq)$ and $D:=\sum_{i=1}^r(d_i-1)$.

From (\ref{eq: estimate rat points CML}) we obtain the following
result.
\begin{theorem}\label{theorem: estimate chi(a,r)}
With notations and assumptions as above, for $d-s+1\le r\le d$ we
have
$$
\left|\chi_r^{\bfs{a}}-\frac{q^{d-s}}{r!}\right|\le
\frac{r(r-1)}{2r!}\,\delta_rq^{d-s-1} +\frac{14}{r!} D_r^3
\delta_r^2(q+1)q^{d-s-2},$$
where $D_r:= \sum_{j=d-r+1}^s(j-1)$ and
$\delta_r:=\prod_{j=d-r+1}^sj=s!/(d-r)!$.
\end{theorem}
\begin{proof}
First we obtain an estimate on the number of $q$--rational points of
$V_r^{\bfs{a}}$. Let
$V_{r,\infty}^{\bfs{a}}:=\mathrm{pcl}(V_r^{\bfs{a}})\cap\{X_0=0\}$.
Combining Theorems \ref{theorem: Vf at infinity} and \ref{theorem:
proj closure of Vf is abs irred} with (\ref{eq: estimate rat points
CML}) we obtain
\begin{eqnarray*}
\big||\mathrm{pcl}(V_r^{\bfs{a}})(\fq)|-p_{d-s}\big|&\le& 14 D_r^3
\delta_r^2q^{d-s-1},  \\[0.5ex]
\big||V_{r,\infty}^{\bfs{a}}(\fq)|-p_{d-s-1}\big|&\le& 14 D_r^3
\delta_r^2q^{d-s-2}.
\end{eqnarray*}
As a consequence,
\begin{eqnarray}
    \big||V_r^{\bfs{a}}(\fq)|-q^{d-s}\big|&\!\!\!\! = &\!\!\!\!
    \big||\mathrm{pcl}(V_r^{\bfs{a}})(\fq)|-|V_{r,\infty}^{\bfs{a}}(\fq)|-
      p_{d-s}+p_{d-s-1}\big|\nonumber\\[1ex]
      &\!\!\!\! \le &\!\!\!\!
      \big||\mathrm{pcl}(V_r^{\bfs{a}})(\fq)|-p_{d-s}\big|+
      \big||V_{r,\infty}^{\bfs{a}}(\fq)|-p_{d-s-1}\big|
      \nonumber\\[1ex] &\!\!\!\! \le  &\!\!\!\! 14 D_r^3
\delta_r^2(q+1)q^{d-s-2}. \label{eq: estimate Vr todos}
  \end{eqnarray}

Next we obtain an upper bound on the number of $q$--rational points
of $V_r^{\bfs{a}}$ which are not useful for our purposes, namely
those with at least two distinct coordinates taking the same value.

Let $V_{r,=}^{\bfs{a}}(\fq)$ be the subset of $V_r^{\bfs{a}}(\fq)$
consisting of all such points, namely
$$V_{r,=}^{\bfs{a}}(\fq):=\bigcup_{1\le i<j\le r}
V_r^{\bfs{a}}(\fq)\cap\{X_i=X_j\},$$
and set $V_{r,\not=}^{\bfs{a}}(\fq):=V_r^{\bfs{a}}(\fq)\setminus
V_{r,=}^{\bfs{a}}(\fq)$. Let $\bfs{x}:=(x_1,\ldots, x_r)\in
V_{r,=}^{\bfs{a}}(\fq)$. Without loss of generality we may assume
that $x_{r-1} = x_r$ holds. Then $\bfs{x}$ is a $q$--rational point
of the affine variety $W_{r-1,r}\subset \{X_{r-1}=X_r\}$ defined by
the polynomials $S_{d-s}^{\bfs{a}}(\Pi_1^*\klk\Pi_s^*)\klk
S_{r-1}^{\bfs{a}}(\Pi_1^*\klk\Pi_s^*)\in\fq[X_1,\ldots X_{r-1}]$,
where $\Pi_i^*:=\Pi_{i}(X_1,\ldots, X_{r-1},X_{r-1})$ is the
polynomial of $\fq[X_1,\ldots, X_{r-1}]$ obtained by substituting
$X_{r-1}$ for $X_r$ in the $i$th elementary symmetric polynomial of
$\fq[X_1,\ldots, X_r]$. Taking into account that
$\Pi_1^*\klk\Pi_s^*$ are algebraically independent elements of
$\cfq[X_1\klk X_{r-1}]$, we conclude that
$S_{d-s}^{\bfs{a}}(\Pi_1^*\klk\Pi_s^*)\klk$ $
S_{r-1}^{\bfs{a}}(\Pi_1^*\klk\Pi_s^*)$ form a regular sequence of
$\fq[X_1,\ldots X_{r-1}]$. This implies that $W_{r-1,r}$ is of
dimension $d-s-1$, and hence, \cite[Proposition 12.1]{GhLa02} or
\cite[Proposition 3.1]{CaMa07} show that
$$|W_{r-1,r}(\fq)|\le \deg W_{r-1,r} q^{d-s-1}\le
\deg V_r^{\bfs{a}} q^{d-s-1}.$$
As a consequence, we obtain
$$|V_{r,=}^{\bfs{a}}(\fq)|\le\frac{r(r-1)}{2}\,\delta_r\,q^{d-s-1}.$$

Combining (\ref{eq: estimate Vr todos}) with this upper bound we
have
$$\big||V_{r,\not=}^{\bfs{a}}(\fq)|-q^{d-s}\big|\le
\frac{r(r-1)}{2}\,\delta_r\,q^{d-s-1} +14 D_r^3
\delta_r^2(q+1)q^{d-s-2}.$$
From this inequality we easily deduce the statement of the theorem.
\end{proof}

The estimate of Theorem \ref{theorem: estimate chi(a,r)} is the
essential step in order to determine the behavior of the average
value set $\mathcal{V}(d,s,\bfs{a})$. More precisely, we have the
following result.
\begin{corollary}\label{coro: average value sets}
With assumptions and notations as in Theorem \ref{theorem: estimate
chi(a,r)}, we have
\begin{equation}\label{eq: estimate V(d,s,a)}
\left|\mathcal{V}(d,s,\bfs{a})-\mu_d\, q-\frac{1}{2e}\right|\le
\frac{s^2+1}{(d-s-1)!}+\frac{21}{8}\frac{s^6(s!)^2}{d!}
\sum_{k=0}^{s-1}\binom{d}{k}\frac{1}{k!}+\frac{7}{q}.
\end{equation}
%
\end{corollary}
\begin{proof}
According to Theorem \ref{theorem: interpolation problem - reduction
to interp sets}, we have
\begin{equation}\label{eq: average value set th 2}
\mathcal{V}(d,s,\bfs{a})-\mu_d\, q=
\sum_{r=1}^{d-s}(-q)^{1-r}\!\!\left(\!\binom {q} {r}
-\frac{q^r}{r!}\!\right)
+\frac{1}{q^{d-s-1}}\hskip-0.25cm\sum_{r=d-s+1}^{d} \hskip-0.25cm
(-1)^{r-1}\!\!\left(\chi_r^{\bfs{a}}-\frac{q^{d-s}}{r!}\!\right).
\end{equation}

First we obtain an upper bound for the absolute value $A(d,s)$ of
the first term in the right--hand side of (\ref{eq: average value
set th 2}). For this purpose, given positive integers $k,n$ with
$k\le n$, we shall denote by \stirling{n}{k} the unsigned Stirling
number of the first kind, namely the number of permutations of $n$
elements with $k$ disjoint cycles. The following properties of the
Stirling numbers are well--known (see, e.g., \cite[\S A.8]{FlSe08}):
$$\stirling{r}{r}=1,\ \stirling{r}{{r-1}}=\binom{r}{2},\
\sum_{k=0}^r\stirling{r}{k}=r!.$$
%

Taking into account the identity
$\binom {q}
{r}=\sum_{k=0}^r\frac{(-1)^{r-k}}{r!}\stirling{r}{k}q^k,$
we obtain
\begin{eqnarray*}
A(d,s):=\sum_{r=2}^{d-s}(-q)^{1-r}\left(\binom {q} {r}\!
-\frac{q^r}{r!}\right)&\!\!\!\!=&\!\!\!\!
\sum_{r=2}^{d-s}q^{1-r}\!\sum_{k=0}^{r-1}
\frac{(-1)^{k+1}}{r!}\stirling{r}{k}q^k\\
&\!\!\!\!=&\!\!\!\!\sum_{r=0}^{d-s-2} \frac{(-1)^r}{2r!}+
\sum_{r=2}^{d-s}q^{1-r}\sum_{k=0}^{r-2}
\frac{(-1)^{k+1}}{r!}\stirling{r}{k}q^k.
\end{eqnarray*}
In order to bound the second term in the right--hand side of the
previous expression, we have
$$
\sum_{k=0}^{r-2} \frac{1}{r!}\stirling{r}{k}q^k\le\!
\sum_{k=0}^{r-3}
\frac{1}{r!}\stirling{r}{k}q^k\!+\frac{1}{r!}\stirling{r}{r\!-\!2}\!q^{r-2}
\!\le q^{r-3}\!+\frac{8}{r^2}q^{r-2}\!\le\!
\left(\frac{1}{d}+\!\frac{8}{r^2}\right)q^{r-2}.
$$
As a consequence, we obtain
$$\left|A(d,s)-\frac{1}{2e}\right|\le \frac{1}{2\cdot(d-s-1)!}+
\sum_{r=2}^{d-s}\left(\frac{1}{d}+
\frac{8}{r^2}\right)\frac{1}{q}\le
\frac{1}{2\cdot(d-s-1)!}+\frac{7}{q}.$$
Next we consider the absolute value of the second term in the
right--hand side of (\ref{eq: average value set th 2}). From Theorem
\ref{theorem: estimate chi(a,r)} we have that
\begin{eqnarray*}B(d,s)&:=&\frac{1}{q^{d-s-1}}\sum_{r=d-s+1}^{d}
\left|\chi_r^{\bfs{a}}-\frac{q^{d-s}}{r!}\right|\\
&\le&\sum_{r=d-s+1}^{d}\frac{r(r-1)}{2r!}\,\delta_r
+\sum_{r=d-s+1}^{d}\frac{14}{r!} D_r^3
\delta_r^2\left(1+\frac{1}{q}\right).
\end{eqnarray*}
Concerning the first term in the right--hand side, we see that
\begin{eqnarray*}\sum_{r=d-s+1}^{d}\frac{r(r-1)}{2r!}\,\delta_r
&\!\!\!\!=&\!\!\!\!
\frac{s!}{2(d-2)!}\sum_{r=d-s+1}^{d}\binom{d-2}{r-2}\\&\!\!\!\!\le&
\!\!\!\!\frac{s\cdot
s!}{2(d-2)!}\binom{d-2}{s-1}=\frac{s^2}{2(d-s-1)!}.
\end{eqnarray*}
On the other hand,
$$\sum_{r=d-s+1}^{d}\frac{14}{r!} D_r^3 \delta_r^2\le
\frac{7}{4}\!\sum_{r=d-s+1}^{d}\!\frac{s^3(s-1)^3(s!)^2}{r!((d-r)!)^2}
=\frac{7}{4}\sum_{k=0}^{s-1}\frac{s^6(s!)^2}{(d-k)!(k!)^2}.$$
Therefore, we obtain
$$B(d,s)\le \frac{s^2}{2(d-s-1)!}+\frac{21}{8}\frac{s^6(s!)^2}{d!}
\sum_{k=0}^{s-1}\binom{d}{k}\frac{1}{k!}.$$
Combining the upper bounds for $A(d,s)$ and $B(d,s)$ the statement
of the corollary follows.
\end{proof}
%
%
\subsection{On the behavior of (\ref{eq: estimate V(d,s,a)})}
In this section we analyze the behavior of the right--hand side of
(\ref{eq: estimate V(d,s,a)}). Such an analysis consists of
elementary calculations, which shall only be sketched.

Fix $k$ with $0\le k\le s-1$ and denote
$h(k):=\binom{d}{k}\frac{1}{k!}$. Analyzing the sign of the
differences $h(k+1)-h(k)$ for $0\le k\le s-2$, we deduce the
following remark, which is stated without proof.
\begin{remark}\label{rem: growth h(k)}
Let $k_0:=-1/2+\sqrt{5+4d}/2$. Then $h$ is a unimodal function in
the integer interval $[0,s-1]$ which reaches its maximum at $\lfloor
k_0\rfloor$.
\end{remark}

From Remark \ref{rem: growth h(k)} we see that
\begin{equation}\label{eq: expresion a analizar}
\frac{s^6(s!)^2}{d!} \sum_{k=0}^{s-1}\binom{d}{k}\frac{1}{k!}\le
\frac{s^7(s!)^2}{d!}\binom{d}{\lfloor k_0\rfloor}\frac{1}{\lfloor
k_0\rfloor!}=\frac{s^7(s!)^2}{(d-\lfloor k_0\rfloor)!(\lfloor
k_0\rfloor!)^2}.
\end{equation}
In order to obtain an upper bound for the right--hand side of
(\ref{eq: expresion a analizar}) we shall use the Stirling formula
(see, e.g., \cite[p. 747]{FlSe08}): for $m\in \N$, there exists
$\theta$ with $0\le \theta<1$ such that $m!=(m/e)^m\sqrt{2\pi
m}\,e^{\theta/12m}$ holds.

Applying the Stirling formula, and taking into account that
$2(s+1)\le d$, we see that there exist $\theta_i$ ($i=1,2,3$) with
$0\le \theta_i<1$ such that
$$C(d,\!s)\!:=\!\frac{s^7(s!)^2}{(d\!-\!\lfloor k_0\rfloor)!(\lfloor
k_0\rfloor!)^2}\le \frac{(\frac{d}{2}-1)^8(\frac{d}{2}-1)^{d-2}
\,e^{2+\lfloor
k_0\rfloor+\frac{\theta_1}{3d-6}-\frac{\theta_2}{12(d-\lfloor
k_0\rfloor)}-\frac{\theta_3}{6\lfloor k_0\rfloor}}}{\big(d-\lfloor
k_0\rfloor\big)^{d-\lfloor k_0\rfloor}\sqrt{2\pi(d-\lfloor
k_0\rfloor)}\lfloor k_0\rfloor^{2\lfloor k_0\rfloor+1}}.$$
By elementary calculations we obtain
\begin{eqnarray*}
 (d-\lfloor k_0\rfloor)^{-d+\lfloor k_0\rfloor}&\le &
 d^{-d+\lfloor k_0\rfloor}e^{{\lfloor k_0\rfloor(d-\lfloor k_0\rfloor)}/{d}}, \\
 \frac{ d^{\lfloor k_0\rfloor}}{{\lfloor k_0\rfloor}^{2\lfloor k_0\rfloor}} &\le &
 e^{(d-\lfloor k_0\rfloor^2)/\lfloor k_0\rfloor},\\
 \left(\frac{d}{2}-1\right)^{d-2}&\le &\left(\frac{d}{2}\right)^{d-2}e^{4/d-2}.
\end{eqnarray*}
It follows that
$$C(d,s)\le {\Big(\frac{d}{2}-1\Big)^8}\frac{e^{\lfloor k_0\rfloor+
\frac{1}{3d-6}+\frac{4}{d}+{\frac{\lfloor k_0\rfloor}{d} (d-\lfloor
k_0\rfloor)}+\frac{1}{\lfloor k_0\rfloor}(d-\lfloor
k_0\rfloor^2)}}{d^22^{d-2}\sqrt{2\pi(d-\lfloor k_0\rfloor)}\lfloor
k_0\rfloor}.$$
By the definition of $\lfloor k_0\rfloor$, it is easy to see that
$$\begin{array}{rcl}
\lfloor k_0\rfloor+\frac{\lfloor k_0\rfloor}{d} (d-\lfloor
k_0\rfloor) &\le & 2\lfloor k_0\rfloor-\frac{1}{5},
\\[1ex]
 \frac{1}{\lfloor k_0\rfloor}(d-\lfloor
k_0\rfloor^2)  &\le&4,\\[1ex]
\frac{(\frac{d}{2}-1)^3}{d^2\lfloor k_0\rfloor\sqrt{d-\lfloor
k_0\rfloor}}&\le&\frac{3}{20}.
\end{array}$$
Therefore, taking into account that $d\ge 2$, we conclude that
\begin{equation}\label{eq: expresion cota}
C(d,s)\le \frac{3(\frac{d}{2}-1)^5e^{\frac{1}{3d-6}+\frac{4}{d}-
\frac{1}{5}+3+\sqrt{5+4d}}}{5\,\sqrt{2\pi}\,2^d}.
\end{equation}
Combining this bound with Corollary \ref{coro: average value sets}
we obtain the main result of this section.
\begin{theorem}\label{theorem: final main result}
With assumptions and notations as in Theorem \ref{theorem: estimate
chi(a,r)}, we have
\begin{equation}\label{eq: upper bound error term}
\left|\mathcal{V}(d,s,\bfs{a})-\mu_d\, q-\frac{1}{2e}\right|\le
\frac{(d-2)^5e^{2\sqrt{d}}}{2^{d-2}} +\frac{7}{q}.
\end{equation}
\end{theorem}
\begin{proof}
From (\ref{eq: expresion cota}) and the fact that $\sqrt{5+4d}\le
4/5+2\sqrt{d}$ holds for $d\ge 2$, we conclude that
$$\frac{21}{8}\frac{s^6(s!)^2}{d!}
\sum_{k=0}^{s-1}\binom{d}{k}\frac{1}{k!}\le
3\,\frac{(d-2)^5e^{2\sqrt{d}}}{2^d}.$$
On the other hand, it is not difficult to see that
$$\frac{s^2+1}{2(d-s-1)!}\le \frac{(d-2)^5e^{2\sqrt{d}}}{2^d}.$$
From these inequalities the statement of the theorem easily follows.
\end{proof}

We make several remarks concerning the upper bound of (\ref{eq:
upper bound error term}).
\begin{remark}
Let $f:\Z_{\ge 4}\to\R$, $f(d):=e^{2\sqrt{d}}(d-2)^52^{-d}$. Then
$f$ is a unimodal function which reaches its maximum value at
$d_0:=14$, namely $f(d_0)\approx 1.08\cdot 10^5$. Furthermore, it is
easy to see that $\lim_{d\to+\infty}f(d)=0$, and indeed for $d\ge
51$, we have $f(d)< 1$.

An obvious upper bound for the left--hand side of (\ref{eq: upper
bound error term}) is
$|\mathcal{V}(d,s,\bfs{a})-\mu_d\, q-(2e)^{-1}|\le (1-\mu_d)q.$
Direct computations show that the upper bound of Theorem
\ref{theorem: final main result} is not interesting for small values
of $q$ if $d\le 44$.

On the other hand, with a slight further restriction for the range
of values admissible for $s$, namely for $1\le s\le \frac{d}{2}-3$,
it is possible to obtain significant improvements of the upper bound
of Theorem \ref{theorem: final main result}. More precisely, arguing
as in the proof of Theorem \ref{theorem: final main result} we
obtain the upper bound
\begin{equation}\label{eq: upper bound restricted error term}
\left|\mathcal{V}(d,s,\bfs{a})-\mu_d\, q-\frac{1}{2e}\right|\le
\frac{9(d-6)e^{2\sqrt{d}}}{2^{d-2}} +\frac{7}{q}.
\end{equation}
Let $g:=\Z_{\ge 7}\to\R$, $g(d):=9(d-6)e^{2\sqrt{d}}2^{-d+2}$. Then
$g$ is a unimodal function reaching at $d_1:=9$ its maximum value,
namely $g(d_1):=85$. Furthermore, we have that
$\lim_{d\to+\infty}g(d)=0$ and $g(d)<1$ for $d\ge 24$. In
particular, (\ref{eq: upper bound restricted error term}) is
nontrivial for $d\ge 19$.
\end{remark}
\begin{remark}
It may be worthwhile to discuss the asymptotic behavior of the
right--hand side of (\ref{eq: estimate V(d,s,a)}). Let
$$H(d,s):=\frac{s^6(s!)^2}{d!}
\sum_{k=0}^{s-1}\binom{d}{k}\frac{1}{k!}.$$
Let $a_d(k):=\binom{d}{k}\frac{1}{k!}$ for $0\le k\le d$. In
\cite{LiPi81} it is shown that $a_d$ is a unimodal function in the
integer interval $[0,d]$ reaching its maximum at $\lfloor
k_0\rfloor$, where $k_0$ is defined as in Remark \ref{rem: growth
h(k)}. Furthermore, for $\epsilon>1/4$ it is proved that
$$\sum_{k=0}^da_d(k)\sim\sum_{k\in(k_0-d^\epsilon,k_0+d^\epsilon)}a_d(k)
\sim\frac{1}{2\sqrt{\pi e}}d^{-1/4}e^{2\sqrt{d}},$$
where the symbol $\sim$ denotes equal asymptotic behavior. Assume
that $s>\lfloor k_0\rfloor+d^\epsilon$ with $\epsilon>1/4$. Then by
the Stirling formula we obtain
$$H(d,s)\sim \frac{1}{\sqrt{2e}}\left(\frac{e}{d}\right)^d\left(\frac{s}{e}\right)^{2s}s^7
e^{2(s-\sqrt{d})}d^{-3/4}.$$
We finally observe that, if $s\le \lfloor k_0\rfloor+d^\epsilon$
with $\epsilon>1/4$, then the right--hand side of this expression is
an upper bound for $H(d,s)$ for $d$ sufficiently large. This shows
that $H(d,s)$ converges to 0 with a double exponential rate
$d^{-(1-2\lambda)d}$ for $s\le \lambda d$ with $\lambda\in[0,1/2[$.
\end{remark}
%








\end{document}
